\documentclass[11pt]{article}

\usepackage{amsmath,amsfonts,amssymb, amsthm}
\usepackage{graphicx}
\usepackage{fancyhdr}
\usepackage{color}
\usepackage{mathrsfs, enumerate}

\usepackage{caption} 

\newtheorem{theorem}{Theorem}[section]
\newtheorem{proposition}{Proposition}[section]
\newtheorem{lemma}{Lemma}[section]
\newtheorem{definition}{Definition}[section]

\newtheorem{example}{Example}[section]

\newtheorem{remark}{Remark}[section]

\newcommand{\bfT}{\mathbf{T}}

\newcommand{\bff}{\mathbf{f}}

\newcommand{\bfy}{\mathbf{y}}

\newcommand{\bfzer}{\mathbf{0}}

\newcommand{\bbA}{\mathbb{A}}

\newcommand{\bbF}{\mathbb{F}}

\newcommand{\bbN}{\mathbb{N}}

\newcommand{\bbR}{\mathbb{R}}

\newcommand{\cF}{\mathcal{F}}

\newcommand{\cP}{\mathcal{P}}

\newcommand{\cR}{\mathcal{R}}
\newcommand{\cS}{\mathcal{S}}

\newcommand{\scrH}{\mathscr{H}}

\newcommand{\al}{\alpha}
\newcommand{\bet}{\beta}
\newcommand{\del}{\delta}
\newcommand{\Del}{\Delta}
\newcommand{\gam}{\gamma}

\newcommand{\eeps}{\epsilon}
\newcommand{\eps}{\varepsilon}
\newcommand{\lam}{\lambda}
\newcommand{\Lam}{\Lambda}

\newcommand{\Om}{\Omega}

\newcommand{\sig}{\sigma}

\newcommand{\txtin}{\quad\text{in}\quad}

\newcommand{\txton}{\quad\text{on}\quad}
\newcommand{\txtor}{\quad\text{or}\quad}
\newcommand{\txtfor}{\quad\text{for}\quad}
\newcommand{\txtand}{\quad\text{and}\quad}

\newcommand{\txtwith}{\quad\text{with}\quad}

\newcommand{\txtforall}{\quad\text{for all}\quad}
\newcommand{\txtforany}{\quad\text{for any}\quad}

\newcommand{\txtae}{\quad\text{for a.e.}\quad}

\newcommand{\cl}{\overline}
\newcommand{\half}{\frac{1}{2}}

\newcommand{\texthalf}{{\textstyle\frac{1}{2}}}

\newcommand{\dom}{\mathscr{D}}

\newcommand{\tl}{\tilde}

\newcommand{\p}{\partial}

\newcommand{\into}{\hookrightarrow}

\newcommand{\dfn}{\mathrel{\mathop:}=}

\newcommand{\intT}{\int_0^T}

\newcommand{\intt}{\int_0^t}

\newcommand{\intOm}{\int_{\Omega}}

\newcommand{\ltom}{{L^2(\Omega)}}

\newcommand{\fmbox}[1]{\fbox{$\displaystyle {#1}$}}

\renewcommand{\implies}{\Rightarrow}

\newcommand{\const}{\text{const}\,}

\renewcommand{\div}{\operatorname{div}}

\newcommand{\curl}{\operatorname{curl}}

\renewcommand{\lg}{\langle}

\newcommand{\rg}{\rangle}

\newcommand{\vecn}[1]{\begin{bmatrix} #1 \end{bmatrix}}

\newcommand{\ds}{\displaystyle}

\def\system{\eqref{wave}--\eqref{ic}}
\def\odesystem{\eqref{ode}--\eqref{ode:ic}}

\newcommand{\V}{H^{1}_0(\Om)}
\newcommand{\lp}{\boldsymbol{\big[\!\!\!\;\big[}}
\newcommand{\rp}{\boldsymbol{\big]\!\!\!\;\big]}}

\begin{document}
\label{pageinit}

\date{}

\title{Stability Analysis of Degenerately-Damped Oscillations}

\author{
Thomas Anderson$^1$, George Avalos$^2$, 
Elizabeth Galvin$^2$,
Ian Kessler$^3$,\\
Michelle Kleckner$^4$,
Daniel Toundykov$^{2, \star}$,
William Tritch$^5$
}

\maketitle

\noindent {$^1$ Department of Computing \& Mathematical Sciences, California Institute of Technology, CA 91125
}\\
\noindent  $^2$ Department of Mathematics, University of Nebraska-Lincoln, NE 68588\\
\noindent  $^3$ Carroll College,  MT 59625\\
\noindent  $^4$ School of Public Health, University of Michigan,  MI 48104\\
\noindent  $^5$ Department of Mathematics \& Statistics, Texas Tech University, TX 79409\\

\noindent $^\star$ \emph{Corresponding author:} dtoundykov@unl.edu


\markboth{T. Anderson, et al.}{Degenerately-Damped Wave Equation}

\begin{abstract}
Presented here is a study of well-posedness and asymptotic stability of a ``degenerately damped"  PDE  modeling a vibrating elastic string. The coefficient of the damping may vanish at small amplitudes thus weakening the effect of the dissipation. It is shown that the resulting dynamical system has strictly monotonically decreasing energy and uniformly decaying lower-order norms, however,  is \emph{not uniformly stable} on the associated finite-energy space. These theoretical findings were motivated by numerical simulations of this model using a finite element scheme and successive approximations. A description of the numerical approach and sample plots of energy decay are supplied. In addition, for certain initial data the  solution can be determined in closed form up to a dissipative nonlinear ordinary differential equation. Such solutions can be used to assess the accuracy of the numerical examples.
\end{abstract}

\section{Introduction}
Advances in nonlinear functional analysis and the
 rich theory of linear distributed parameter systems have led to a growing of body of work  on nonlinear infinite-dimensional models.
For  instance, in a 2nd-order evolution  framework (especially, wave, elastodynamics, or thin plates with no rotational inertia terms) for an appropriate elliptic operator $A$ a linear equation with viscous damping for an unknown $u=u(t,x)$ may be expressed as
\[
  \ddot{u} + A(x) u + \bet(x)\dot{u} =F(t,x)
\]
with $\bet>0$. We will focus on the evolution on a bounded domain and under suitable homogeneous boundary conditions.  A nonlinear refinement on the dissipative  term may take the form of a  feedback law $g(\dot{u})$. Stability properties of such models have been extensively analyzed.  In an infinite-dimensional sitting such a nonlinear feedback may change the topology of the problem and uniform stability becomes reliant on the regularity of solutions (for example, see \cite{van-mar:00, nicaise:03, las-tou:06}).

A more general scenario would account for coefficients that depend on the solution itself:
\[
  \ddot{u} + A(x) u + \bet(x, u, \dot{u})\dot{u} =F(t,x)\,.
\]
Assuming the well-posedness of an associated initial-boundary value problem can be resolved, if the term $\bet(x,u,\dot{u})$  is not guaranteed to be strictly positive on a \emph{fixed} appropriately configured set, then analysis of stability becomes much more involved  since the  region where the dissipation is active  now evolves with the solution and may not always comply with the requirements of the geometric optics.
 The case when this coefficient vanishes at zero displacement, namely, $\bet(x,0,\dot{u})=0$,  will be referred to as \emph{degenerate} damping.

Such a degeneracy naturally arises when investigating energy decay of higher-order norms. For example, the natural energy space for a semilinear wave problem
\[
 \ddot{u} - \Del u + g(\dot{u}) =0
\]
 is $u \in W^{1,2}(\Om)$ and $\dot{u}\in \ltom$. With more regular initial data one can consider behavior of higher-order energy norms, namely for $(u,u_t)\in W^{2,2}(\Om)\times W^{1,2}(\Om)$. One approach  would be  to differentiate the PDE in time which via the substitution $v= \dot{u}$  leads  to a degenerately damped problem
 \[
  \ddot{v} - \Del v +  g'(v) \dot{v} = 0
 \]
A particular example can be observed in  the relation between  Maxwell's system and the (vectorial) wave equation. For a given medium denote the electric permittivity  by $\eeps$, magnetic permeability by $\lam$  and conductivity by $\sig$. Then Maxwell's system reads
\begin{eqnarray*}
	\dot{E} - \curl(\lam H) + \sig E &=& 0  \\
    \dot{H} + \curl(\eeps E) &=& 0\,,
\end{eqnarray*}
with $\div (\lam H)\equiv 0$. On a bounded domain, subject to the \emph{electric wall} boundary conditions, and for scalar-valued $\lam$, $\eeps$, $\sig$ with  positive lower-bounds, the term  $\sig E$ exponentially stabilizes this system \cite{nic-pig:05:AAA}. In a more accurate nonlinear conduction model the coefficient  $\sig$ may depend on the intensity of the electric field $E$. If we consider, for example, $\sig = \al |E|^p$ for $p\geq 1$, then differentiating the first equation in time and combining with the equation for $H$ gives
\[
  \ddot{E}+ \curl(\lam \curl \eeps E) + \al p |E|^{p-2}E\cdot \dot{E} E + \al |E|^p \dot{E} = 0\,.
\]
For example, taking, $p=1$ gives
\[
  \ddot{E}+ \curl(\lam \curl \eeps E) + \al (E\cdot\dot{E})\hat{E} + \al |E| \dot{E} = 0
\]
where $\hat{E}$ is the normalized vector $E$. The term  $\al |E| \dot{E}$ has features of the viscous dissipation in this second-order equation, but nonlinear conductivity augments it with a degenerate coefficient $\al |E|$.

\medskip
The study of stability for the above models is much more delicate than in the situations where the damping, even nonlinear, depends on the time-derivative only.
Weighted energy methods---from basic energy  laws to Carleman estimates  (e.g. \cite{las-tri:97,tri-yao:02:AMO, las-tri-zha:04:JIIP-I,las-tri-zha:04:JIIP-II,tri-xu:07:AMS, chu-las-tou:08:DCDS, buc-tou:10})---have been successfully used to derive stabilization and observability inequalities for distributed parameter systems. 
However,  these methods typically rely on the properties of the coefficients to ensure that suitable geometric optics conditions are satisfied  and the control effect suitably ``propagated"  \cite{bar-leb-rau:92} across the physical domain.
One can sometimes dispense with geometric optics requirements for  smooth enough initial data \cite{leb-rob:97:DMJ,leb-rob:95:CPDE, bellassoued:05, bor-tom:10:MA, ell-tou:12:EECT}, yet even then the support of the control/damping term must contain a subset that is time-invariant  (and with any time-dependent coefficients being non-vanishing, e.g., as in \cite{bellassoued:05}). In turn,
the analysis of control-effect propagation when the coefficients themselves depend on the solution and possibly go to zero wherever and whenever the solution does would require new techniques.

\subsection{The model}
The following semilinear model, if recast in a higher-dimensional setting becomes highly non-trivial even when just regarding  local wellposedness.  In a one-dimensional framework the nonlinearity is more tractable, but the rigorous stability analysis has long been open. We focus on an elastic string with a \emph{degenerate damping}, namely a dissipative term whose coefficient depends on and may vanish with  the amplitudes:
\begin{equation}\label{wave}
  \ddot{u} - u_{xx}  + f(u)\dot{u} = 0,\txtfor x\in \Om\dfn (0,1),\quad  t>0
\end{equation}
fixed at the end-points
\begin{equation}\label{bc}
 u(t,0) \equiv u(t,1) \equiv 0 \txtfor t \geq 0
\end{equation}
and with a prescribed initial configuration at $t=0$:
\begin{equation}\label{ic}
 u(0,x) =u_0(x),\quad \dot{u}(0,x) = u_1(x) \quad\text{for a.e.}\quad x\in(0,1)\,.
\end{equation}
The initial data $u_0, u_1$ live in the natural function spaces revisited below.  Function $f$ is assumed to be continuous non-negative, hence the term $f(u)\dot{u}$ a priori should provide some form of energy dissipation in the model.

The scenario of interest is when $f(s) \to 0$ as $s\to 0$, essentially causing the  dissipative effect  to deteriorate at small amplitudes. We will focus on the polynomial case
\begin{equation}\label{as:f}
 f(s) \dfn  \al s^{2m},\quad \al>0,\quad  m\in \bbN\,.
\end{equation}
satisfying the locally Lipschitz estimate
\begin{equation}\label{f:factor}
  f(s) - f(r) =  M(s,r)\cdot (s-r) \txtwith  M(s,r)= \al \sum_{j=0}^{2m-1} s^j r^{2m-1-j}\,.
\end{equation}


\subsection{Known results and new challenges}\label{sec:challenges}

Existence and uniqueness of weak \emph{finite-energy} solutions to \system{} was proven in  \cite{ram-stre:02:TAMS} by means of Galerkin approximations. The advantage of a 1D framework is that the displacement function is absolutely continuous, hence topologically $f(u)\dot{u}$ is still in $L^2(0,1)$, as in the case of the corresponding linear model. 
However, in higher-dimensional analogues this embedding property is lost and proving existence becomes a markedly more complex task. First, fractional damping exponents were considered in order  to ensure that the damping term is bounded with respect to the finite energy topology \cite{pit-ram:02:IUMJ}. Arbitrary damping exponents were subsequently  examined in  \cite{bar-las-ram:05,bar-las-ram:07:IUMJ}. Due to the loss of regularity solutions had to be characterized via a variational \emph{inequality} and established by a rather technical application of Kakutani's fixed point theorem.

On the other hand, stability analysis even in one dimension  poses a challenge that has been open for a number of years.
Despite the gain in regularity, attributed to Sobolev embeddings, the key difficulty now is that energy estimates require some sort of information on the region where the damping is supported. In \eqref{wave} both the magnitude and the support of the damping coefficient evolve with the geometry of the state, rendering all standard techniques inapplicable.

It is plausible to assume that some sort of a logarithmic uniform decay rate can be verified, possibly by combining ODE techniques (e.g. \cite{las-tat:93, ram-tou-wil:12:DCDS}) with pointwise Carleman-type estimates. Another, though a rather weak, tentative indication of this outcome would be the uniform stability of the corresponding finite-dimensional analogue (see the Appendix). Yet the situation in infinite dimensions turns out rather different.

\subsection{Contribution of this work}
The goal of this article is to  examine analytically and numerically stability properties of the dynamical system associated to \system{}:
\begin{itemize}
	\item Establish global persistence of regularity in solutions with smooth initial data. Besides  theoretical interest such a result is useful to justify the convergence estimates for numerical approximations.

	\item Prove  that a polynomial degeneracy in the damping of the form \eqref{as:f} yields a system that is \textbf{not uniformly stable}.
	
	\item  Present a numerical scheme that indicates the loss in decay rates. Such  observations had been performed first and, in fact, served as a motivation for the theoretical results presented here.
\end{itemize}

\subsection{Outline}
The notation employed throughout the paper is summarized in Section \ref{sec:notation}. The two main results on well-posedness and stability are stated in Section \ref{sec:main}.

Several auxiliary technical definitions used in the proofs can be found in Section \ref{sec:aux}.
Local and global wellposedness are verified respectively in Sections \ref{sec:local} and \ref{sec:global}. They draw upon two regularity lemmas proved earlier in Section \ref{sec:regularity}.

The proof of the lack of uniform stability is the subject of Section \ref{sec:nonuniform}.  Numerical results are the subject of Section \ref{sec:numerics}.

The Appendix contains results pertaining to the ODE analog of the considered problem, namely, a damped harmonic oscillator with the damping coefficient dependent on the displacement.

\section{Preliminaries}

\subsection{Notation}\label{sec:notation}
This section serves as a quick reference for the basic notation used thorough the paper with some of the symbols revisited and discussed in more detail later in the text.

Henceforth $\|\cdot\|_X$ will denote the norm on a normed space $X$. For the space $\ltom$ we will use
\[
|u|_{_0}\dfn \|u\|_{\ltom},
\]
with the corresponding inner product denoted by $(\cdot, \cdot)_{_0}$.  We will also frequently involve the Sobolev space 
\[
H^1_0(\Om)\dfn  W^{1,2}_0(\Om)
\]
associated to an equivalent inner product and norm
\[
 (u,v)_{_1} = (u_x,v_x)_{_0}\, \quad |u|_{_1} \dfn \sqrt{(u,u)_{_1}}\,.
\]
The bilinear form $\lg \cdot,\, \cdot\rg$ will indicate the  pairing of $\V$ and its continuous dual $H^{-1}(\Om)$.


We will also frequently use spaces of the form
\[
 C^n([0,T];X) \txtor L^p(0,T; X),
\]
which will be abbreviated respectively as
\[
  C^n_T X  \txtand L_T^p X\,.
\]
Looking ahead, for the one-dimensional Dirichlet Laplacian operator $A$ (discussed below) let us introduce the space
\begin{equation}\label{def:STn}
 S_T^n \dfn  \left\{ z   \;\mid\;  z\in C_T^j(\dom(A^{\frac{n+1-j}{2}}))  \txtfor j=0,1,\ldots,n+1 )\right\}\,
\end{equation}
equipped with the natural graph norm.  For example, $S_T^0 = C([0,T]; \dom(A^{1/2}))
\cap C^1([0,T]; \ltom)$ indicates the standard regularity in space and time for a weak solution to a linear wave equation. In turn $S_T^1= C_T\dom(A)\cap C^1_T \dom(A^{1/2})\cap C^2_T \ltom$ does the same for a strong solution to such an equation.

We will also be using spaces
\[
 \scrH^n\dfn \dom(A^{\frac{n+1}{2}})\times\dom(A^\frac{n}{2})\,.
\]
Thus $\scrH^0$ is the natural finite energy state space for a linear wave problem and $\scrH^1$ denotes the domain of the corresponding evolution generator.

Relation $a \lesssim b$ will occasionally be used to indicate that $a\leq C b$ for a constant $C$ which only depends on the main parameters of the system, e.g., the size of the domain $\Om$ or the exponent $p$ of the coefficient in \eqref{as:f}.

\subsection{Laplace operator}
For convenience let us summarize some of the fundamentals. Consider the operator
\begin{equation}\label{def:A}
A = -\p_{xx} : \dom(A)\subset \ltom\to \ltom,\quad \dom(A) = W^{2,2}(\Om)\cap H_0^1(\Om)
\end{equation}
which is positive, self-adjoint with compact resolvent, and has eigenvalues 
\[
\lam_n = n^2 \pi^2,\quad n\in \bbN
\]
with the corresponding eigenfunctions
\[
 E_n(x) = \sqrt{2}\sin(n\pi x),\quad n\in \bbN\,.
\]
The $E_n$'s form an orthonormal basis for $\ltom$.
For $r\in \bbR$  we can define fractional powers of $A$:
\[
 A^r \left(\sum_{n=1}^\infty c_n E_n\right) = 
 \sum_{n=1}^\infty \lam_n^r c_n E_n
\]
with
\[
\dom(A^r) = \left\{ \sum_{n=1}^\infty c_n E_n \quad: \quad   
\sum_{n=1}^\infty   |\lam_n|^{2r} |c_n|^2 <\infty
\right\}\,.
\]
The eigenfunctions $\{E_n\}$ form an  orthogonal basis for every $\dom(A^r)$, $r\in \bbR$. Some of the fractional powers can be identified with Sobolev spaces, e.g.
\[
 \dom(A^{1/2}) = \V \txtand \dom(A^{-1/2}) = [\dom(A^{1/2})]' = H^{-1}(\Om)\,.
\]
Since in our situation the model is one-dimensional, then trivially no issues in these identifications arise in regard to the regularity of the domain. Operator $A$ also corresponds to the Riesz isomorphism $\V \to H^{-1}(\Om)$, and for $u,v \in \V$ we have
\[
 \lg A u, v \rg = (u,v)_{_1},.
\]

\subsection{State space and linear group generator}
The natural finite-energy state space associated with the evolution driven by  \eqref{wave}--\eqref{bc} is
\[
 \scrH^0 \dfn \dom(A^{1/2}) \times\ltom\,.
\]
If we set $y  = (u,\dot{u})$ we can recast this problem as an evolution equation
\[
 y' = \bbA y  
\]

for the skew-adjoint operator $\bbA :\dom(\bbA)\subset \scrH^0\to \scrH^0$
\[
 \bbA = \begin{bmatrix}
 0  &  I\\
 -A  &  0 \\ 
\end{bmatrix}
\]
with 
\[
 \dom(\bbA) = \dom(A) \times \dom(A^{1/2})\,.
\]
We will also consider smoother solutions for which we define
\begin{equation}\label{def:Xn}
\scrH^n \dfn \dom(\bbA^n) = \dom(A^{\frac{1+n}{2}})\times \dom(A^{n/2}) \txtfor n\in \bbN\,
\end{equation}
with the associated graph norm given via
\begin{equation}\label{norm-cHm}
\begin{split}
\| (u_0, u_1)\|_{\scrH^n}^2 =&
\| u_0\|_{\dom(A^{(n+1)/2})}^2 + \| u_1\|_{\dom(A^{n/2})}^2 \\
=&
\sum_{k=1}^\infty \lam_k^n\left(\lam_k |\cF_k[u_0]|^2 + 
 |\cF_k[u_1]|^2\right)\,.
\end{split}
\end{equation}
Here $\cF_k$ denotes the $k$-th Fourier coefficient with respect to the Hilbert basis of $L^2(\Om)$ given by the eigenfunctions $E_k$ of $A$. Note also that in one-dimension the following continuous injection holds
\begin{equation}\label{inclusions}
 \scrH^n \subset  H^{n+1}(\Om) \times H^n(\Om) \into
 C^n(\cl{\Om}) \times C^{n-1}(\cl{\Om})
\end{equation}
for $n\in \bbN\cup\{0\}$ if we adopt the notation convention $C^{-1}(\cl{\Om}) \dfn \ltom$.

\section{Results}

We start with a formalized notion of a solution to the class of PDE systems of the form \system{}. The following slightly more abstract formulation will help streamline the subsequent discussion.

\subsection{Auxiliary definitions}\label{sec:aux}
\begin{definition}[Wave problem]
Let $\lp F, \phi_0, \phi_1\rp$  be the shorthand for the initial-boundary value problem:
\[
  \ddot{\phi} - \phi_{xx}  = F,\txtfor x\in \Om,\; t>0\,,
\]
with the indicated derivatives taken in the sense of distributions, and subject to boundary conditions
\[
  \phi(t,0) \equiv 0 \equiv \phi(t, 1) \txtfor t>0
\]
and initial data
\[
 \phi(0,x) = \phi_0(x), \quad \dot{\phi}(0,x)= \phi_1(x)\txtae x\in \Om\,.
\]

\end{definition}

\begin{definition}[Weak solution of linear problem]\label{def:weak}
Suppose $(u_0,u_1)\in\scrH^0$ and for some $T>0$, $F\in L_T^1\ltom$. Then we say a function
\begin{equation}\label{weak-regularity}
u\in C^1_T \ltom \cap C_T \V \quad \left( \text{equivalently } u\in S_T^0 \text{ or }  (u,\dot{u})\in C_T \scrH^0\right)
\end{equation}
is a weak solution to $\lp F, u_0, u_1\rp$ on interval $[0,T]$  if
\begin{enumerate}[(i)]
\item  $u(0, x) = u_0(x) \txtand \dot{u}(0,x) = u_1(x) \txtae x\in \Om$,
\item For any $\phi\in \ltom$ the scalar map    $t\mapsto (\dot{u}(t), \phi)_{_0}$
is absolutely continuous (hence a.e. differentiable) on $[0,T]$.

\item For any $\phi \in \V$
\begin{equation}\label{weak-d-dt}
 \frac{d}{dt}  ( \dot{u}(t), \phi)_{_0} +( u(t), \phi)_{_1}=
 \big( F(t),  \phi \big)_{_0}  \quad\text{for a.e.}\quad t\in (0,T)\,.
\end{equation}

\end{enumerate}
\end{definition}

\begin{definition}[``Regular" functions]
\label{def:regular}
A function $u$  on $[0,T]\times \Om$ will be described as \textbf {regular of order} $n\in \bbN\cup \{0\}$ if  it is continuously differentiable in time with the following regularity
\begin{equation}\label{def:regular-n}
 (u,\dot{u})\in C_T \scrH^n\,.
\end{equation}
In classical terminology, weak solutions correspond to order $0$ and  strong solutions  to order $1$.

\end{definition}

Suppose  $u=u(t,x)$ has the  weak regularity \eqref{weak-regularity} (regular of order $0$). Then
according to the (1D) Sobolev embedding  $W^{m,2}(\Om) \into C^{m-1}(\cl{\Om}) $ for $ m\in \bbN$,  the function $F(t,x) = f(u(t,x))\dot{u}(t,x)$ is well-defined as an element of $L^2_T\ltom$. In fact we will generalize this statement for the purposes of subsequently analyzing more regular solutions.

\begin{proposition}\label{prop:F-regularity}
Let   $f(s) =\al s^{2m}$ for $m\in \bbN$. If   $z\in S_T^n$, then
\begin{equation}\label{f-z-dot-z-regularity}
 A^{s/2} [ f(z)\dot{z}] \in  C_T^j \dom(A^{\frac{n-s-j}{2}}), \txtfor  s,j\in \bbN\cup\{0\},\quad s+j\leq n\,.
\end{equation}
In addition,
\begin{equation}\label{p_t-f-bound}
\| \p_t^j ( f(z)\dot{z})  \|_{C_T\ltom} \leq \cP\left[  \|\p_t^k z\|_{C _T\V} \right]_{k=1,\ldots, j-1} (1 + \|\p_t^j \dot{z}\|_{C_T\ltom}),\quad j\leq n,
\end{equation}
where $\cP$ is a polynomial in  $j-1$ variables.
\end{proposition}
\begin{example}\label{example-1}
Due to a variety of spaces and indices involved in the statement of Proposition \ref{prop:F-regularity}, it is helpful to look at a basic example. Take $\al = m=1$, so $f(z)\dot{z}= z^2\dot{z}$ and consider the regularity order $n =4$. Then the condition $z\in S^n_T$ reads
\begin{equation}\label{ex1:z}
z \in  \bigcap_{j=0}^{5}  C_T^{j}(\dom(A^{\frac{5-j}{2}})) 
\end{equation}
In particular, $z\in C_T \dom(A^{5/2})$, which corresponds to $5$ square-integrable derivatives, 
first three of which satisfy zero boundary conditions. We have, for example,
\begin{eqnarray*}
 %
    \p_t^4 [f(z)\dot{z}] &=& 
  30  \dot{z} \ddot{z}^2 
+  20  \dot{z}^2 \dddot{z}  
+   20 z \, \ddot{z} \, \dddot{z}
+   10 z \dot{z} \p_t^4z
      + z^2  \fmbox{\p_t^4 \dot{z}}\,.
\end{eqnarray*}
Thus, for instance, $| \p_t^4 f(z)\dot{z}|_0$ can be estimated using a polynomial of $L^\infty(\Om)$ bounds on the functions $z$, $\dot{z}$, \ldots, $\p_t^4 z$, and one term involving the $\ltom$ norm of the fifth derivative of $z$ in time or, equivalently,  $| \p_t^4 \dot{z}|_0$.  This is precisely the conclusion of  \eqref{p_t-f-bound}.

Likewise, if we consider, say, the $2$-nd derivative in space and $2$-nd in time to $s=2$, $j=2$ in  \eqref{f-z-dot-z-regularity} we get:
\begin{equation}\label{d-t-2}
\p_t^2 [f(z)\dot{z}]
=2\dot{z}^3+6z \dot{z} \ddot{z}+z^2 \dddot{z}
\end{equation}
\begin{equation}\label{d-x-2-d-t-2}
\begin{split}
\p_x^2 \p_t^2 [f(z)\dot{z}]=&
12\dot{z}_x^2\dot{z}+12z_x\ddot{z}\dot{z}_x+12z_x\dot{z}\ddot{z}_x+2z_x^2\dddot{z}+4z\dddot{z}_xz_x+12z\ddot{z}_x\dot{z}_x+6\dot{z}z_{xx}\ddot{z}+6\dot{z}^2\dot{z}_{xx}\\
&+2z\dddot{z}z_{xx}+6z\ddot{z}\dot{z}_{xx}+6z\dot{z}\ddot{z}_{xx}+z^2\fmbox{\dddot{z}_{xx}}
\end{split}
\end{equation}
We have that $\p_t^2 [f(z)\dot{z}]$ has zero trace as follows from \eqref{d-t-2} and the zero boundary condition on $z$.
Moreover, the highest-order  term $\dddot{z}_{xx}$  in $\p_x^2 \p_t^2 [f(z)\dot{z}]$ can  be bounded in $\ltom$ since by \eqref{ex1:z} we have $\dddot{z}\in \dom(A)$. The rest of the terms are in fact in $L^\infty(\Om)$. This confirms that 
$A\p_t^2 [f(z)\dot{z}]\in C_T \ltom$ in agreement with \eqref{f-z-dot-z-regularity}.

Note that if we had, say, $n=6$ then in the same context we would need to prove that $A \p_t^2[f(z)\dot{z}] \in C_T \dom(A)$ instead of just $C_T \ltom$. First of all, $n=6$ would give $z\in C_T \dom(A^{7/2})$ implying that $z, z_x,\ldots, \p_x^5 z$ have zero traces. Hence so do their time-derivatives and then it is immediately follows that  $A \p_t^2[f(z)\dot{z}]$ satisfies zero trace condition. And taking two more derivatives in space in \eqref{d-x-2-d-t-2} yields the highest-order term $\p_x^4 \p_t^3z$ which is in $\ltom$ from  $z\in S_T^6$ that implies $\dddot{z}\in C_T\dom(A^2)$. Thus, $\p_x^4 \p_t^3z\in \ltom$ 
and  whereas other summands in  $\p_x^2 A \p_t^2[f(z)\dot{z}] $ belong to  $L^\infty(\Om)$. We conclude that if $n=6$, then $A \p_t^2[f(z)\dot{z}]$ (is in $\ltom$) has zero trace and integrable second derivative, again in accordance with \eqref{f-z-dot-z-regularity}. 
\end{example}

\begin{proof}
Let 
\[
F(z)\dfn f(z)\dot{z}
\]
and first let's show that $\p_t^j F$ belongs to $C_T \dom(A^{(n-j)/2})$.  For $j\leq n$,  $\p_t^j F$ is a polynomial in the following variables
\[
\p_t^j F =  \cP[f^{(k)}(z),\; \p_t^k z,\; \p_t^k \dot{z}]_{k=0,1,\ldots,j }
\]
that is affine with respect to the highest-order derivative $\p_t^j\dot{z} = \p_t^{j+1} z$, which is at least in $C_T\ltom$ by the assertion that $z\in S_T^n$ (recall \eqref{def:STn} and plug $j=n+1$). We can just bound coefficients of $\p_t^{j+1} z$ using the fact that $\V$ embeds continuously into $C(\cl{\Om})$. Thus the bound \eqref{p_t-f-bound} follows. 

Next, $\dom(A^{\frac{n+1-j}{2}})$ embeds into $C^{n-j}(\cl{\Om})$ for $j \leq n$, so all of the terms in $\p_x^{n-j}\p_t^j F$ are $C_T C(\cl{\Om})$ except for the highest order term  $\p_x^{n-j}\p_t^j \dot{z} \in C_T\ltom$. Since $\p_x^{n-j}\p_t^j F$ is affine with respect to  that term with continuous coefficients, then
\[
 \p_x^{n-j}\p_t^j F \in C_T \ltom\,.
\]
It follows that
\[
\p_t^j F \in C_T W^{n-j,2}(\Om) \txtfor j\leq n\,.
\]

To strengthen this regularity to $C_T \dom(A^{\frac{n-j}{2}})$ we must verify the boundary conditions. Since $\dom(A)$ coincides with the $W^{2,2}(\Om)$ functions that have zero trace, then it is sufficient to show that 
\begin{equation}\label{zero-trace}
\p_x^k \p_j F   = 0  \txton \p\Om\, \txtfor
0\leq k + j\leq n-2\,\quad (k,j \geq0)
\end{equation}
That is, we can show that  every derivative of total order (time $+$ space) up to $n-2$  of $F$ vanishes on the boundary.   The asserted regularity $z\in S^{n}_T$ implies that  $\p_x^j \p_t^k z$ has zero trace for any $j\leq n-1$ and any $k$. Since any  $(n-2)$-order (space and time together) derivative of $F$ involves at most $(n-1)$-st order terms in $z$, then  \eqref{zero-trace} readily follows. Thus
\[
\p_t^j F \in C_T\dom(A^{\frac{n-j}{2}})\,.
\]
Because  $A^{s/2}$ is by definition a bounded operator on $\dom(A^{s/2})$, then for $\frac{s}{2}\leq \frac{n-j}{2}$ we have
\[
  A^{s/2}\p_t^j F = \p_t^j A^{s/2} F \in C_T \dom(A^{\frac{n-j-s}{2}})\,.
\]
confirming \eqref{f-z-dot-z-regularity}.
\end{proof}

\medskip

Relying on (a special case of) Proposition \ref{prop:F-regularity}  we  formulate the notion of solution to \system{}.
\begin{definition}[Weak solution to  \system{}]
We say $u$ is a weak solution to \system{} on $[0,T]$ if  it is a weak solution to $\lp F, u_0, u_1\rp$ with $F(t,x)= f(u(t,x))\dot{u}(t,x)$ (which is in $L_T^2\ltom$ by Proposition \ref{prop:F-regularity} for $n=s=j=0$).
\end{definition}

\begin{definition}[Energy]\label{definition-energy}
For a function $z=z(t,x)$  define the quadratic energy functional of order $n$ by
\begin{equation}\label{def:energy}
  E^{(n)}_z(t)\dfn \half |\p_t^n z(t)|_{_1}^2 + \half |\p_t^n\dot{z}(t)|_{_0}^2 \txtwith E_z \dfn E^{(0)}_z\,.
\end{equation}
\end{definition}

\subsection{Main theorems}\label{sec:main}
With the above definitions in mind, the new results of interest are:
\begin{theorem}[\bf Global well-posedness of weak and regular solutions]\label{thm:well-posed}
Let $f$ be as in \eqref{as:f} with exponent $2m$, $m\in \bbN$. Suppose for some integer $n\geq 0$  the initial condition \eqref{ic} satisfies 
\[
(u_0,u_1) \in \scrH^n \dfn \dom(A^{(n+1)/2}) \times \dom(A^{n/2})\,.
\]
Then there exists a unique weak solution  $u$, regular of order $n$ (Definition \ref{def:regular}), of \system{} on $[0,T]$ for any $T>0$. Moreover  $u\in S_T^n$, that is, $u\in C_T^j \dom(A^{\frac{n+1-j}{2}})$ for $j=0,1,\ldots,n.$
\end{theorem}
\begin{example}\label{ex:square}
If a solution to \system{} with $f(u)=u^2 \dot{u}$ has initial data given by the eigenfunctions of $A$ (every $\dom(A^r)$), then  $u\in C^\infty([0,\infty) ; W^{p,\infty}(\Om))$ for every $p\geq 1$.
\end{example}

\medskip

\begin{theorem}[\bf Lack of uniform stability]\label{thm:unstable}
The dynamical system generated by  \system{} on the state space corresponding to weak solutions $\scrH^0\dfn \V \times \ltom$
is non-accretive, but is  \textbf{not} uniformly stable.  
Specifically, the energy functional $t\mapsto E_u(t)$ is continuous non-increasing; however, for any constants $0<r<\cl{r}$  and any time $\cl{T}>0$ there exists an initial datum $(\cl{u}_0, \cl{u}_1)\in B_{\cl{r}}(0)$ such that the corresponding solution trajectory on $[0,\cl{T}]$ does not intersect $B_{r}(0)$.  

However, for any $s\in(0,1)$ the lower-order norm $\|u(t)\|_{W^{s,2}(\Om)}$ decays to zero as $t\to \infty$ with the decay rate uniform with respect to bounded sets of initial data. In addition, $E_u(t)$ is strictly monotonically decreasing on any interval where $E_u(t)$ is positive.
\end{theorem}
\begin{remark}[Open problem]
The question of \emph{strong} stability of the dynamical system generated by \system{} on the finite energy space $\scrH^0$  remains open, as far as we are aware. That is, for any given solution can we prove that $\lim_{t\to \infty} E_u(t)=0$?
\end{remark}

\section{Well-posedness}

\subsection{Linear problem}

For convenience let us summarize a few classical results that can be easily verified, for example, using separation of variables.
\begin{lemma}\label{lem:linear}
Consider the problem $\lp F\equiv 0, u_0, u_1\rp$ with $(u_0,u_1)\in \scrH^0$. Then there exist a group $t\mapsto \cS(t)$ of linear operators on  $\scrH^0$  such that $(u(t),\dot{u}(t)) \dfn \cS(t) (u_0,u_1)$  determines a weak solution to this problem on every  $[0,T]$, $T>0$. The group is explicitly given by
\begin{equation}\label{semigroup}
\cS(t)(u_0,u_1) \dfn \sum_{k=1}^\infty \vecn{\ds 
\cos(\sqrt{\lam_k} t) &   \frac{1}{\sqrt{\lam_{k}}}\sin(\sqrt{\lam_k} t)   \\[.2in]
\ds  
-\sqrt{\lam_k} \sin(\sqrt{\lam_k} t) & \cos(\sqrt{\lam_k} t) \\[.05in]
}\vecn{\cF_k[u_0]\\[.2in]\cF_k[u_1]} : \vecn{E_k \\[.2in] E_k}
\end{equation}
where $\{(\lam_k,E_k)\}$ are the eigenvalue-eigenfunction pairs for the operator $A$ defined in \eqref{def:A}, and
 $\cF_k$ is the $k$-th Fourier coefficient with respect to the Hilbert basis $\{E_k\}$ for  $\ltom$:
\[
 E_k(x) = \sqrt{2}\sin(k\pi x)  \txtand \lam_k = k^2\pi^2\,.
\]
Moreover,  $\cS(t)$ is a unitary operator on the space $\dom(\bbA^k)$ with respect to the norm given by \eqref{norm-cHm}, as follows by direct calculation using \eqref{semigroup}. 
\end{lemma}

The next result is likewise well-known:
\begin{proposition}[Inhomogeneous linear problem]\label{prop:linear-problem}
Consider the problem $\lp F, w_0,  w_1 \rp$ with
\begin{equation}\label{wave:lin:ic}
(w_0,  w_1) \in \scrH^0 \,.
\end{equation}
If $F\in L^2_T \ltom$, this problem possesses a unique weak solution   $w\in S_T^0$. Then the continuous mapping $t\mapsto E_w(t)$ for energy functional  \eqref{def:energy}  satisfies
\begin{equation}\label{energy-identity:0}
 E_w(t) + \intt (F(s), \dot{w}(s))_{_0} \,ds = E_w(0) \txtforall t\in [0,T]\,.
\end{equation}
In particular, from the Gronwall estimate it readily follows
\[
 \|(w(t), \dot{w}(t))\|_{\scrH^0}^2 \leq   C\left(\|(w_0, w_1)\|_{\scrH^0}^2 + \|F\|^2_{L^2_T\ltom}\right) e^t  \txtforall t\in [0,T]\,.
\]
Moreover, if  $(w_0,w_1)\in \scrH^1$ and $F\in C^1_T\ltom$, then $w \in C^2_T\ltom$ and $w_{xx}\in C_T\ltom$  (e.g. see \cite[Thm. 2.1, p. 229]{b:barbu93:rom}\footnote{There's a minor misprint in \cite[Thm. 2.1, p. 229]{b:barbu93:rom}: the first assumption is meant to read $u_0\in H_0^1(\Om)$ (instead of ``$H^2_0(\Om)$"). In the second half of that theorem, which is the one we cite, it is strengthened to $u_0\in H_0^1(\Om)\cap H^2(\Om)$ for strong solutions.}) 

\end{proposition}
\qed


\subsection{Variational formulation}

\begin{proposition}[Variational form]
Suppose $u$ is a weak solution of \system{} on $[0,T]$. Then for any test-function $v\in C^1_t\ltom\cap C_t\V$ with $t\in [0,T]$, it satisfies the variational identity
\begin{equation}\label{variational}
 ( \dot{u}, v)_{_0} \bigg|_{_0}^t - \intt ( \dot{u}, \dot{v})_{_0} + \intt ( u, v)_{_1}  + \intt ( f(u)\dot{u}, v )_{_0} = 0
\end{equation}
\end{proposition}
\begin{proof}
Let $c \in C^1([0,T])$, then for any $\phi \in \V$ we have from \eqref{weak-d-dt}
\begin{equation}\label{variational-truncated}
 ( \dot{u}(t), c(t) \phi)_{_0} \bigg|_{_0}^t  -\intt  ( \dot{u}(t), c'(t)\phi )_{_0}  dt +  \intt (u (t), c(t)\phi)_{_1}  dt + 
\intt \left( f(u(t)) \dot{u}(t), c(t) \phi \right)_{_0}   dt= 0\,.
\end{equation}
Let $\{E_n\}$ be the orthonormal basis for $\ltom$  consisting of the eigenfunctions of $A$.  
Given $v\in C^1_t \ltom\cap C_t \V$ we can represent it as
\[
v(s,x) = \sum_{k=1}^\infty c_k(s)e_k(x)
\]
for $c_k\in C^1([0,T])$, $k\in \bbN$. This series
that converges to $v$ in $C_t \V$  and  its time-derivative $\sum_{k=1}^\infty c_k'(s)e_k(x)$ converges to $v'$ in $C_t\ltom$.

By applying identity \eqref{variational-truncated} to finite sums $\sum^M c_k e_k$ and passing to the limit $M\to \infty$ we recover \eqref{variational}. 
\end{proof}

\subsection{Relation between regularity and higher-order energy}\label{sec:regularity}

The existence of finite energy solutions to \system{} is known. Well-posedness for \emph{regular solutions}, however, requires more work and relies on the connection between the smoothness in space and smoothness in time as  summarized by the diagram below:
\[
  E^{(n)}_u(t)\dfn \half |\p_t^n u(t)|_{_1}^2 + \half |\p_t^n\dot{u}(t)|_{_0}^2\,.
\]
\begin{center}
\begin{tabular}{ccc}
local solutions in $C_T \scrH^n$  &   $\longrightarrow$  &  $E_u^{(n)}(t)$  is well defined\\
$\uparrow$ && $\downarrow$\\
bound on the $\scrH^n$ norm & $\longleftarrow$ & Energy identity and a bound on $E_u^{(n)}(t)$ 
\end{tabular}
\end{center}

The  purpose of this subsection is to furnish this connection which can be loosely outlined as follows: a weak solution $u$ of \system{} is regular of order $n$ on $[0,T]$ if and only if the $n$-th order energy is bounded on $[0,T]$. We split this claim into two propositions.

In order not to keep track of how the structure of $f(u)\dot{u}$ changes after differentiation we introduce the following, somewhat abstract property:
\begin{definition}[Regularity dependence]\label{def:reg-dependence}
We say two functions $(z,F) $  have \textbf{order $n$ dependence} if for every $k=1,2,\ldots n$ the regularity (\textbf{if} it holds)
\[
  z \in C_T^j \dom(A^{\frac{n+1-j}{2}}) \txtfor j\leq k 
\]
\textbf{would imply} that
\[
F \in C_T^j\dom(A^{\frac{n-j}{2}}) \txtfor j\leq k
 \]
where $j$ is a  non-negative integer. 
It is helpful to note:
\begin{itemize}

	\item if $(z,F)$ have order $n$ dependence, they trivially have order $\ell$ dependence for any $\ell=0,1,2,\ldots,n$.

	\item  if $(z,F)$ have order $n$ dependence,  then  for $s\in \bbN$, $s\leq n$ the functions $A^{s/2}z$ and $A^{s/2}F$ have order $(n-s)$ dependence.

\end{itemize}
\end{definition}

Again, it is helpful to consider an example.
\begin{example}
Let $F = f(z)\dot{z}$. Then it is not hard to check that $(z,F)$ have order $n$ dependence. For instance take $n=4$ and assume $z\in C_T^j \dom(A^{\frac{5-j}{2}})$.    As was shown via \eqref{d-t-2} and \eqref{d-x-2-d-t-2} in Example \ref{example-1}
\[
 \p_t^2 F  \in \dom(A) 
\]
continuously in time. In particular, $F\in C^2_T \dom(A^\frac{4-2}{2})$.
\end{example}

\begin{proposition}\label{prop:regular-differentiable}Suppose  that $u$ is a weak solution on $[0,T]$ to $\lp F, u_0, u_1\rp$ with $F\in C_T\ltom$ and $\|(u,\dot{u})\|_{C_T\scrH^0} \leq R_1$. Assume functions 
$(u, F)$ have order $n$ dependence for some $n\in \bbN\cup\{0\}$.

Suppose, in addition,
\begin{equation}\label{En-regularity}
(\p_t^n u, \p_t^n \dot{u})\in C_T \scrH^0
\end{equation}
 with $\|(\p_t^n u,\p_{t}^{n}\dot{u})\|_{C_T\scrH^0} \leq R_2 $, then
\begin{equation}\label{u:regularity}
u\in S_T^n, 
 \end{equation}
in other words, $u\in C_T^j \dom(A^\frac{n+1-j}{2}) \txtfor j\leq n+1$. Moreover,
 \begin{equation}\label{regular-bound}
  \|(u,\dot{u})\|_{S_T^n} \leq  K(n, R_1,R_2),
 \end{equation}
 for a constant $K(n, R_1,R_2)$ dependent only on $n$, $R_1$, $R_2$ and  being continuous monotone increasing with respect to the parameters $R_1$, $R_2$.
 \end{proposition}
\begin{proof}
Throughout the argument below, the norms in the considered spaces can be (inductively) estimated in terms of $R_1$ and $R_2$, thus ultimately verifying \eqref{regular-bound}. We will  focus in detail on proving the claimed regularity.

\textbf{Cases $n=0$, $n=1$.}
By assumption we always have $(u,\dot{u})\in C_T\scrH^0$ and $F\in C_T \ltom$ which takes care of the case $n=0$.
If we in addition assume  $(\dot{u},\ddot{u})\in C_T \scrH^0$, then the equation $\ddot{u}-u_{xx} = F$ implies that $u_{xx}\in C_T\ltom$. Since $u=0$ on the boundary then $u\in C_T\dom(A)$. Conclude:
\[
  (u,\dot{u})\in S^1_T\,.
\]
For $n\geq 2$ proceed by induction: suppose the result of this Proposition holds for $0,1,\ldots, n-1$. Assume \eqref{En-regularity} holds. 
Then let us show  \eqref{u:regularity}. 

\textbf{Case $n\geq 2$.}
Because $(u,\dot{u})\in C_T\scrH^0$, then condition \eqref{En-regularity} implies $(u,\dot{u})\in C_T^n\scrH^0$. As a special case (using $n-1$ instead of maximal $n$), we have
\[
(\p_t^{n-1} u, \p_t^{n-1}\dot{u})\in C_T\scrH^0\,.
\]
Moreover,  functions $(u,F)$ have a fortiori  $(n-1)$-st order dependence, so the induction hypothesis gives
\[
 u\in S_T^{n-1}\,.
\]

Next, by assumption \eqref{En-regularity}, we also have
\begin{equation}\label{first-part}
\p_t^{n} \dot{u}  = \p_t^{n+1} u   \in C_T\ltom\,.
\end{equation}
%
To show that $u\in S^n_T$, it remains to verify that for $j=0,\ldots, n$ we have $\p^j_t u \in C_T \dom(A^{\frac{n+1-j}{2}})$. To this end introduce
\[
 z= A^{1/2}u \txtand \tl{F} = A^{1/2}F.
\]
Since $(u,F)$ have $n$-th order dependence, then $z$ and $\tl{F}$ have $(n-1)$-st order dependence (see Definition \ref{def:reg-dependence}).
As we already know, $u\in S_T^{n-1}$, and via the $(n-1)$-st order dependence of $(u,F)$, we have $F \in C_T^j\dom(A^{\frac{n-1-j}{2}})$ for $j\leq n-1$.  Because $n\geq 2$,  then $A^{1/2} F \in C_T\ltom$; likewise $n\geq 2$ also tells us that $u$ is at least in $C_T\dom(A)$, so $A^{1/2}u_{xx} = A^{3/2} u = \p_{xx} A^{1/2}u \in C_T \dom(A^{-1/2}) $. 

Consequently, applying $A^{1/2}$ to the equation for $u$, we conclude that  $z=A^{1/2}u$ is a weak solution to $\lp \tl{F}, z(0), \dot{z}(0) \rp$ with $\tl{F}\in C_T\ltom$ and $(z(0),\dot{z}(0))\in \scrH^0$.  By  \eqref{En-regularity} we also have 
\[
 (\p_t^{n-1} z, \p_t^{n-1} \dot{z}) \in C_T \scrH^0\,.
\]
The induction hypothesis now states that 
\[
 z \in S_T^{n-1}  \txtfor j=0,1,\ldots,n\,.
\]
It implies that
\[
 u \in C_T^j \dom(A^\frac{n+1-j}{2}) \txtfor j=0,1,\ldots, n,
\]
as desired. From here, along with \eqref{first-part}, it follows that $u\in S_T^n$ which completes the proof of the  implication ``case $n-1 \implies $ case $n$" for the induction argument.
\end{proof}

The next result complements the previous  proposition, demonstrating that the same regularity $S_T^n$ of solutions can be inferred from a priori differentiability in space, rather than in time.
\begin{proposition}\label{prop:differentiable-regular}
Suppose $u$ is a weak solution on $[0,T]$  of   $\lp F, u_0, u_1\rp$ with $F\in C_T\ltom$. Assume $(u,F)$  have $n$-th order dependence (Definition \ref{def:reg-dependence}) for some $n\in \bbN\cup\{0\}$.
If $u$ is regular of order $n$, i.e., $\| (u,\dot{u})\|_{ C_T \scrH^n}\leq R<\infty$, then
\[
 u\in S_T^n
\]
and 
\begin{equation}\label{bound-energy-n-K}
\|(u,\dot{u})\|_{S_T^n} \leq K(n,R)
\end{equation}
with $K$ dependent only on $n$ and $R$ and continuous monotone increasing with respect to  parameter $R$.
\end{proposition}
\begin{proof}
In the course of the proof, the bound \eqref{bound-energy-n-K} can be traced inductively to depend only on $\|(u,\dot{u})\|_{C_T\scrH^n}$. To keep the exposition concise the  argument will focus on the regularity.

If $n=0$, the claimed regularity simply matches that of weak solutions. For $n=1$ we are given $u\in C_T\dom(A)$ and $\dot{u} \in C_T\dom(A^{1/2})$. Solving  $\ddot{u}= u_{xx} - F \in C_T\ltom$ verifies that $u\in C_T \dom(A)\cap C_T^1 \ltom =  S^1_T$.

Proceed by induction. Fix $n\geq 2$, suppose the statement holds for all $\tl{n}\leq n-1$ and assume  $\|(u,\dot{u})\|_{C_T\scrH^n}$ is finite. A fortiori it follows that $(u,\dot{u})\in C_T\scrH^{n-1}$. Then $u\in S_T^{n-1}$ by the induction hypothesis.

Define   $z\dfn A^{1/2}u$, then using $u\in S_T^{n-1}$ it follows as in the proof of Proposition \ref{prop:regular-differentiable}    that $z$ is a weak solution to $\lp A^{1/2}F, z(0), \dot{z}(0)\rp$ with $A^{1/2}F\in C_T\ltom$. Moreover from the assumption $u\in C_T \scrH^{n}$ we also have that  $z \in C_T\scrH^{n-1}$, i.e., is regular of order $n-1$. Thus by the induction hypothesis
\[
 z \in S_T^{n-1}\,,
\]
equivalently
\begin{equation}\label{deriv-up-to-n}
  u \in C_T^j \dom(A^\frac{n+1-j}{2}) \txtfor j=0,1,\ldots ,n.
\end{equation}
The only remaining step from here to proving $u\in S_T^n$ is  to show that $\p_t^{n+1}u \in C_T\ltom$. To this end define
\[
 w\dfn \p_t^{n-1} u\,.
\]
Then by \eqref{deriv-up-to-n}
\[
 (w,\dot{w}) \in C_T \scrH^1\,.
\]
Since $u\in S_T^{n-1}$ and by  the $(n-1)$-st order  dependence (implied by $n$-th order dependence) of  $u$ and $F$, we deduce from \eqref{deriv-up-to-n}  that  $\p_t^{n-1}F \in C_T\ltom$. So
\[
 \ddot{w} = w_{xx} - \p_t^{n-1} F \in C_T\ltom
\]
confirming that 
\[
\ddot{w} =  \p_t^{n+1} u \in C_T \ltom
\]
as desired. Thus $u\in S_T^n$ completing the induction argument.
\end{proof}

\subsection{Local unique solutions}\label{sec:local}

It is known that  \system{} possesses unique solutions \cite{ram-stre:02:TAMS}. Here we extend this result to regular solutions as well. First we formulate it for local in time solutions.

\begin{theorem}\label{thm:local}
Suppose $\|(u_0,u_1)\|_{\scrH^n} = R<\infty$ for a non-negative integer $n$. 
Then there exists $T=T(R)>0$ such that system \system{} has a local unique solution that is regular of order $n$ on interval $[0,T]$.
\end{theorem}
\begin{proof}
\textbf{Step 1: the spaces.} Note that for $\psi\in \dom(A^{n/2})$
\[
 \| \psi\|_{\dom(A^{n/2})} \quad
 \lesssim\quad |\psi|_{_0} + \left| \p_x^{n}\,\psi\right|_{_0} 
 \,.
\]
Moreover, for $(\psi_0,\psi_1) \in \scrH^n$ we have
\begin{equation}\label{C-Xn-estimate}
\left\|\psi_0\right\|_{C^{n}(\cl{\Om})}
 +\left\|\psi_1\right\|_{C^{n-1}(\cl{\Om})}  + \|\psi_1\|_{W^{n,2}(\Om)} \quad\lesssim\quad    \| (\psi_0,\psi_1)\|_{\scrH^n}
\end{equation}
with the temporary notational convention  $ C^{-1}(\cl{\Om}) \dfn \ltom $.

\textbf{Step 2: contraction mapping.} 
Let $t\mapsto \cS(t)$ be the semigroup for the linear wave equation with data $y_0\in \scrH^n$. Note that if  $z$ is regular of order $n$, then 
by Proposition \ref{prop:F-regularity} we would get
 \begin{equation}\label{thm:local:as:f}
f(z)\dot{z} \in C_T \dom(A^{n/2})\,.
\end{equation}
With this in mind introduce the operator $\Lam_T$ on $C_T\scrH^n$,
\[
 (\Lam_T \phi)(t) \dfn \cS(t)y_0 +\int_0^t \cS(t-s) \vecn{0 \\ f(\phi_0(s)) \phi_1(s)}  ds\quad \text{for any}
 \quad \phi = (\phi_0,\phi_1) \in C_T\scrH^n\,.
\]
For  $t\in [0,T]$ we have
\begin{equation}\label{Lambda-diff}
\begin{split}
\left\| \Lam_T \phi(t) - \Lam_T \tl{\phi}(t) \right\|_{\scrH^n} \leq & \int_0^t \left\| \cS(t-s) \vecn{0 \\ f(\phi_0(s)) \phi_1(s) - f(\tl{\phi}_0(s)) \tl{\phi}_1(s)}  \right\|_{\scrH^n} ds\,.
\end{split}
\end{equation}
From the local Lipschitz property  \eqref{f:factor} of $f$ we obtain
\[
\begin{split}
 f(\phi_0)\phi_1 - f(\tl{\phi}_0)\tl{\phi}_1  = &
 \left[ f(\phi_0)\phi_1 -  f(\tl{\phi}_0)\phi_1\right] + \left[f(\tl{\phi}_0)\phi_1  -f(\tl{\phi}_0)\tl{\phi}_1\right] \\
  =&
 \phi_1 M(\phi_0,\tl{\phi}_0)(\phi_0 - \tl{\phi}_0) + f(\tl{\phi}_0)(\phi_1 - \tl{\phi}_1)\,.\,
\end{split}
\]
Now we will rely on the fact that a priori $\phi,\tl{\phi}\in C_T\scrH^n$ gives:
\[
 \phi_0,\; \tl{\phi}_0 \in C_T C^n(\cl{\Om}), \quad
  \phi_1,\; \tl{\phi}_1 \in C_T C^{n-1}(\cl{\Om}) \txtand  \phi_1, \tl{\phi}_1 \in  C_T W^{2,n}(\Om)\,.
\]
Estimate,
\begin{equation}\label{fixed-pt-1}
\begin{split}
&\left\| f(\phi_0)\phi_1 - f(\tl{\phi}_0)\tl{\phi}_1 \right\|_{\dom(A^{n/2})}\\
\lesssim &
\left| \phi_1 M(\phi_0,\tl{\phi}_0)(\phi_0 - \tl{\phi}_0) + f(\tl{\phi}_0)(\phi_1 - \tl{\phi}_1) \right|_{_0} 
+ \left| \p_x^n\big[ \phi_1 M(\phi_0,\tl{\phi}_0)(\phi_0 - \tl{\phi}_0) + f(\tl{\phi}_0)(\phi_1 - \tl{\phi}_1) \big]\right|_{_0}\\
\leq &
\left|  \phi_1 M(\phi_0,\tl{\phi}_0)(\phi_0 - \tl{\phi}_0) + f(\tl{\phi}_0)(\phi_1 - \tl{\phi}_1) \right|_{_0}\\
&+
\left| \sum_{i+j+k=n} c_{ijk}\,\p_x^i\phi_1 \cdot \p_x^j  M(\phi_0,\tl{\phi}_0)\cdot \p_x^k(\phi_0 - \tl{\phi}_0) \right|_{_0}
+
\left| \sum_{i+j=n} d_{ij}\, \p_x^i f(\tl{\phi}_0)\cdot \p_x^j  (\phi_1 - \tl{\phi}_1)\right|_{_0}\\
\leq & 
\cP_1\left\{ \|\phi_0\|_{C^{n}(\cl{\Om})} , \|\tl{\phi}_0\|_{C^{n}(\cl{\Om})}\right\} \cdot 
\|\phi_1\|_{W^{n,2}(\Om)}
\cdot
\|\phi_0 - \tl{\phi}_0 \|_{C^{n}(\cl{\Om})}
+  \cP_2\left\{ \|\tl{\phi}_0\|_{C^{n}(\cl{\Om})}\right\} \cdot \| \phi_1-\tl{\phi}_1\|_{W^{2,n}(\Om)}
\end{split}
\end{equation}
where $\cP_i\{\cdot\}$ are a  polynomials in the indicated variables with positive coefficients. Now, from \eqref{C-Xn-estimate}
\[
 \|\phi_0-\tl{\phi}_0\|_{C^n(\cl{\Om})} + \|\phi_1-\tl{\phi}_1\|_{W^{2,n}(\Om)} \lesssim \|\phi-\tl{\phi}\|_{\scrH^n}\,.
\]
So continuing \eqref{fixed-pt-1} we get
\[
\begin{split}
\left\| f(\phi_0)\phi_1 - f(\tl{\phi}_0)\tl{\phi}_1 \right\|_{\dom(A^{n/2})}
\leq  & \cP_3 \bigg\{ \|\phi_0\|_{C^{n}(\cl{\Om})}
, \|\tl{\phi}_0\|_{C^{n}(\cl{\Om})}, \|\phi_1\|_{W^{n,2}(\Om)} \bigg\} \|\phi-\tl{\phi}\|_{\scrH^{n}}\\
\leq & \cP_4 \bigg\{ \|\phi\|_{\scrH^{n}},\|\tl{\phi}\|_{\scrH^{n}} \bigg\} \|\phi-\tl{\phi}\|_{\scrH^{n}}\,.
\end{split}
\]

Because the group $\cS(t)$   is unitary on $\dom(\bbA^{n})$  then  \eqref{Lambda-diff} yields
\begin{equation}\label{Lambda-contr}
 \left\| \Lam_T \phi - \Lam_T \tl{\phi} \right\|_{C_T\scrH^n}
 \leq T  \cP_4 \left\{ \|\phi\|_{C_T\scrH^n} , \|\tl{\phi}\|_{C_T\scrH^n}\right\} \|\phi-\tl{\phi}\|_{C_T\scrH^n}\,.
\end{equation}
If $\phi$ and $\tl{\phi}$ come from a bounded set, then choosing $T$ small enough implies that $\Lambda_T$ is a contraction.
 
\textbf{Step 3: invariance.} Finally, we also need   $\Lam_T$ to map the the ball $B_{R}(0)$ in $C_T\scrH^n$ into itself. According to \eqref{Lambda-contr} for that it suffices that $R$ satisfy
\[
  T \cP_4\{R, R\} 2R \leq R\,.
\]
For which we can impose
\begin{equation}\label{T-invariance}
  T <  1/ (2 \cP_4\{R, R\})\,.
\end{equation}
Now the contraction mapping principle implies the claimed local unique solvability.
\end{proof}

\subsection{Energy identities and global existence}\label{sec:global}

The global existence result is based on a priori bounds on the energy.
\begin{proposition}[Extension of local solutions]\label{prop:extension}
Let $f$ be as in \eqref{as:f} with exponent $2m$, $m\in \bbN$. Assume $u$ is a weak solution that is regular of order $n$, defined on some right-maximal interval $I_{max} =[0,T_{max})$.  Suppose   $(u,  \dot{u}) \in C^n_T\scrH^0$ and that  there exists a continuous function $\psi$ on $\bbR_+$ such that 
\[
 E^{(n)}_u(t)\leq \psi(t) \txton I_{max}\,.
\]
Then $T_{\max}=\infty$.
\end{proposition}
\begin{proof}
By Proposition \ref{prop:F-regularity},  $u$ and the term  $F = f(u)\dot{u}$ have $n$-th order dependence. Recall that energy $E^{(n)}_u$ controls the $C_T\scrH^n$ norm of the solution (in fact, it controls the $L^\infty_T\scrH^n$ norm, but the continuity with values in $\scrH^n$ is implied by  the definition of regular solution). Then 
there is a constant $K$ dependent on $ \sup\{ |\psi(t)| : 0\leq t\leq T_{\max}\}$ such that
 \[
\|  (u,\dot{u})\|_{C_T\scrH^n}\leq K
 \]
 
for any $T< T_{\max}$.  Hence by Theorem \ref{thm:local} any initial data of the form $v_0 = u(t)$ and $v_1 = \dot{u}(t)$ for $t\in [0,T_{\max})$ can be extended to a solution that exists for another $\Del T =\Del T(K)$ time units, independently of $t\in [0,T_{\max})$. Hence $T_{\max}$ cannot be finite.
\end{proof}

\begin{lemma}[Energy identity for regular solutions]\label{lem:energy-regular}
Let $f$ be as in \eqref{as:f} with exponent $2m$, $m\in \bbN$. Suppose $u$ is a weak solution to \system{} on $[0,T]$. If $u$ is regular of order $n$, then for $k\leq n$ we have $(u,\dot{u})\in C_T^k\scrH^0$, $f(u)\dot{u} \in C_T^k\ltom$ and 
\begin{equation}\label{k-energy-identity}
  E^{(k)}_u(t) + \intt (\p_t^k [f(u)\dot{u}] ,\p_t^k \dot{u})_{_0} = E^{(k)}_u(0)\,.
\end{equation}

\end{lemma}
\begin{proof}
 Proposition \ref{prop:F-regularity} confirms that $F$ and $F=f(u)\dot{u}$ have $n$-th order regularity dependence. Invoke Proposition \ref{prop:regular-differentiable} to conclude that $u\in S_T^n$. It certainly implies that $(u,\dot{u})\in C^k\scrH^0$ for $k\leq n$. And by the $n$-th order regularity dependence of $(u,F)$ we also have  $F\in C_T^k \ltom$. 

Define
\[
 w = \p_t^k u
\]
then $w$ is a weak solution to $\lp \p_t^k F, w(0), \dot{w}(0) \rp$ with $\p_t^k F\in C_T\ltom$ and $(w(0),\dot{w}(0))\in \scrH^0$.  Now call upon the  energy identity \eqref{energy-identity:0} for $w$ to obtain \eqref{k-energy-identity} for $u$.
\end{proof}

\subsubsection{Finishing the proof: global existence and regularity}

Let $u$ be a weak solution on interval $[0,T]$. Because  $F \cdot \dot{u} \dfn f(u)\dot{u}^2$ is non-negative by assumption \eqref{as:f} on $f$, then in the case
 $n=0$ by Lemma \ref{lem:energy-regular} we have
\begin{equation}\label{bound-energy0}
\texthalf\|(u,\dot{u})\|_{C_T\scrH^0}^2
\leq \sup_{t\in [0,T]} E_u(t) \leq  E_u(0)\,.
\end{equation}
Since the bound is independent of $T\in [0,T_{\max})$, then by Proposition \ref{prop:extension} $u$ extends globally to $t\geq 0$.

Arguing by induction, suppose that $(u,\dot{u})$ is a global solution that is regular of order $n$, and at the same time is \emph{local} regular of order $n+1$ on a maximal interval $[0,T_{\max})$. We want to show that it is also global of order $n+1$.

\begin{example}
It helps to preview the result on the example of $f(s)=s^2$ and by extending from weak to regular of order $1$ solutions.
Suppose $u$ is a global weak solution that is, in addition,  regular of order $1 $ on maximal interval $[0,T_{\max})$. For such a solution by Proposition \ref{lem:energy-regular} we have the ``$1$-st order"  energy identity
\[
 E^{(1)}_u(T)  + \intT \p_t [f(u)\dot{u}] \ddot{u} = E^{(1)}_u(0)\,.
\]
\[
 E^{(1)}_u(T)  + \intT (2u \dot{u}^2  + u^2 \ddot{u})\ddot{u} = E^{(1)}_u(0)\,.
\]
We can bound $C_T L^\infty(\Om)$ norms of $u$ and $\dot{u}$ in terms of the finite energy $E_u(t)$.
The latter is bounded by constant $E_u(0)$  (but, in fact, any continuous on $\bbR_+$ upper bound on $E_u(t)$ would do, which is how the general argument works).  In turn, the term $|\ddot{u}|_{_0}^2$ is a part of the first-order energy. So we have
\[
 E^{(1)}(T) \leq E_u^{(1)}(0) + C(E^{(0)}_u(0)) \intT E^{(1)}_u(t) dt
\]
From here, Gronwall's estimate gives us an asymptotically growing continuous upper bound on  $E^{(1)}_u(t)$.
Then Proposition \ref{prop:extension} ensures that $u$ is regular of order $1$ globally, that is $T_{\max}=\infty$.
\end{example}

Now onto the actual proof of global existence. By the $n$-th order global regularity there is a continuous monotone increasing function $\psi_n:\bbR_+ \to \bbR_+$
\[
\| (u,\dot{u})\|_{S_t^n} \leq \psi_n(t) \txtforall t\geq 0\,.
\]
By $(n+1)$'st order regularity and via  Lemma \ref{lem:energy-regular} for any $T\leq [0,T_{\max})$ we also have
\[
 E^{(n+1)}_u(T) + \int_0^T (\p_t^{n+1} F, \p_t^{n+1}\dot{u} )_0  = E_{u}^{(n+1)}(0)\,.
 \]
Use estimate \eqref{p_t-f-bound} of Proposition \ref{prop:F-regularity} to claim that there is a polynomial $\cP(s)$ such that for each $t$
\[
 |\p_{t}^{n+1} F(t)|_{_0} \leq \cP(\psi_n(t)) (1+|\p^{n+1}_t \dot{u}(t)|_{_0})\,.
\]
Then $ |\p_{t}^{n+1} F(t)|_{_0} |\p_t^{n+1}\dot{u}|_{_0} \leq  
\cP(\psi_n(t))(1+ |\p_t^{n+1}\dot{u}|_{_0}+|\p_t^{n+1}\dot{u}|_{_0}^2)
\leq 2\cP(\psi_n(t))(1+ |\p_t^{n+1}\dot{u}|_{_0}^2)$. In particular, 
\[
 E^{(n+1)}_u(T)   \leq E_{u}^{(n+1)}(0) + 4\int_0^T \cP(\psi_n(t)) (1+E_{u}^{(n+1)}(t)) dt\,.
\]
Gronwall's inequality now provides
\[
 E_u^{(n+1)}(T) \leq \psi_{n+1}(T)\dfn \left(E_{u}^{(n+1)}(0)+ \intT 4\cP(\psi_n(t))dt \right) \exp\left(  \int_0^T 4\cP(\psi_n(t))dt\right)\; \text{ for all }\;  T< T_{\max}\,.
\]
By the assumption on $\psi_n$, the newly defined mapping $\psi_{n+1}$ is  continuous  on  $\bbR_+$. Then Proposition  \ref{prop:extension} ensures that the solution $u$ is global, regular of order $n+1$, which completes the proof by induction.

\bigskip

This step also completes the proof of Theorem \ref{thm:well-posed}. The additional regularity $u\in S_T^n$ follows if we first use  Proposition \ref{prop:F-regularity}, which establishes $n$-th order dependence of $u$ and $f(u)\dot{u}$, and then invoke Proposition \ref{prop:regular-differentiable} which translates $n$-th order regularity into $S_T^n$.

\section{Lack of uniform stability}\label{sec:nonuniform}

This section is devoted to the proof of  Theorem \ref{thm:unstable}. The heart of the argument is that the energy decay slows down with the frequency of solutions, independently of their total finite energy. 

Let $k\in \bbN$ and  consider initial condition of the form
\begin{equation}\label{u0-sine}
 u^{(k)}_0(0,x) = \frac{2}{k\pi} \sin(k\pi x),\qquad  \p_t u^{(k)}_1(0,x) = 0\,.
\end{equation}
Note that $u^{(k)}_0\in \dom(A^r)$ for any $r\in \bbR$.
By direct calculation we have
\begin{equation}\label{l2-norm-u0}
   E_{u^{(k)}}(0) = \texthalf\|\p_x u^{(k)}_0\|^2  + \texthalf \|u^{(k)}_t(0)\|^2 =\texthalf\|\p_x u^{(k)}_0\|^2  = 1 \txtand \|u^{(k)}_0\|= \frac{\sqrt{2}}{k \pi} \,.
\end{equation}
Hence for such initial data, any constant that may be estimated in terms of  $\|(u^{(k)}_0,u^{(k)}_1)\|_{\scrH^0}$ is bounded uniformly in $k $.

\subsection{``Primitive" problem and decay of lower-order norms}
We continue working with the initial data  $(u^{(k)}_0,0)$ as in \eqref{u0-sine}. Pick the anti-derivative through the origin of $f$  (as in \eqref{as:f})
\[
\tl{f}(s) \dfn \frac{s^{2m+1}}{2m+1}\,.
\]
Due to  $|u^{(k)}_0|_{_1}^{2m}$ (hence $\|u^{(k)}_0\|^{2m}_{C(\cl{\Om})}$) being bounded independently of $k $, we have for some constant $C$ \textbf{independent of $k$ in \eqref{u0-sine}} the estimate
\[
| \tl{f}(u^{(k)}_0)|_{_0} \leq  C |u^{(k)}_0|_{_0}\,.
\]

Let function $\Phi^{(k)}$ solve  the linear elliptic boundary value problem
\begin{equation}\label{elliptic}
\left\{
\begin{array}{rcll}
  \p_{xx} \Phi^{(k)}  &=& \tl{f}(u^{(k)}_0) \\
  \Phi^{(k)}(0)&=& \Phi^{(k)}(1) = 0\,.
\end{array}
\right.
\end{equation}
It satisfies the standard elliptic estimate:
\begin{equation}\label{elliptic-phi0}
 |\Phi^{(k)}|_{_1} \leq C' |\tl{f}(u^{(k)}_0)|_{_0} \leq C |u^{(k)}_0|_{_0}\,
 \end{equation}
 for $C'$ and $C$ \textbf{independent of $k$}. We do have $\tl{f}(u^{(k)}_0)\in \ltom$ because $u_0^{(k)}\in \V$. In fact, $u^{(k)}_0$ is an eigenfunction, and it is easy to show  (as in Proposition \ref{prop:F-regularity}) that
\[
 \tl{f}(u^{(k)}_0) \in \dom(A) \cap C^{\infty}(\Om)\,.
\]
We then  have
 \[
 \Phi^{(k)} \in \dom(A^2) \subset W^{4,2}(\Om)\,.
 \]

With this function $\Phi^{(k)}$ at hand, consider the nonlinear ``primitive" problem
\begin{equation}\label{sys:primitive}
  \ddot{\phi}^{(k)} -\phi^{(k)}_{xx} +\tl{f}(\dot{\phi}^{(k)}) = 0 \txtfor x\in \Om, \; t>0
\end{equation}
with homogeneous Dirichlet boundary data
\begin{equation}\label{sys:primitive:bc}
 \phi^{(k)}(t,0)\equiv \phi^{(k)}(t,1)\equiv 0 \txtfor t>0
\end{equation}
and initial condition
\begin{equation}\label{sys:primitive:ic}
 \phi^{(k)}(0) = \Phi^{(k)}  \txtand  \dot{\phi}^{(k)}(0) = u_0\,.
\end{equation}

Note the following  properties of the system \eqref{sys:primitive}--\eqref{sys:primitive:ic}:
\begin{enumerate}[i.)]

\item Because $s\mapsto\tl{f}(s)$ is  continuous, monotone increasing vanishing at $0$, then \eqref{sys:primitive} is a semilinear wave problem with monotone damping. The data $(\Phi^{(k)}, u^{(k)}_0)$ come from the domain of the associated nonlinear generator; in this 1D setting  it coincides with the domain  $\scrH^1$ of the corresponding linear group. This well-posedness result is based on monotone operator theory \cite[Thm.  3.7, p. 306]{b:barbu:93}.

	In particular, we have for any $T>0$
	\[
	\phi^{(k)} \in L^\infty_T \dom(A) \cap C^1_T \dom(A^{1/2})\txtand \ddot{\phi}^{(k)}\in L^\infty_T \ltom\,.
	\]
Moreover, given the 1D embeddings it follows that
	\[
	\tl{f}(\dot{\phi}^{(k)}) \in C_T^1 \ltom \cap C_T \dom(A^{1/2})\,.
	\]
	So $\phi^{(k)}$ satisfies \eqref{sys:primitive} and has  $(\phi(0),\dot{\phi}(0))\in \dom(A^2)\times \dom(A^{3/2})$. By Proposition \ref{prop:linear-problem}  we have
	\[
	\phi_{xx}^{(k)},\; \ddot{\phi}^{(k)} \in C_T \ltom\,.
	\]

\item   The velocity $v^{(k)}(t,x)\dfn  \dot{\phi}^{(k)}(t,x)$  satisfies $(v^{(k)},\dot{v}^{(k)}) \in L^2_T \scrH^0$ and is a  weak  solution to
	\begin{equation}\label{v-problem}
	  \ddot{v^{(k)}}- v^{(k)}_{xx} + f(v^{(k)}) v^{(k)}_t = 0\,.
	\end{equation}
 In addition,
\[
 v^{(k)}(0) = \dot{\phi}^{(k)}(0) = u^{(k)}_0 \in \ltom  \qquad(\text{actually, in any }\dom(A^r))
\]
and because $\phi_{xx}^{(k)}$, $\ddot{\phi}^{(k)}$ are continuous with values in $\ltom$, then
\begin{eqnarray*}
\dot{v}^{(k)}(0) &= & \ddot{\phi}^{(k)}(0) \stackrel{\eqref{sys:primitive}}{=}  \phi^{(k)}_{xx}(0) - \tl{f}(\dot{\phi}^{(k)}) = \p_{x}^2\Phi^{(k)}  - \tl{f}(u^{(k)}_0) =0
\end{eqnarray*}
according to the way $\Phi_0$ was constructed in \eqref{elliptic}. \emph{Thus the primitive problem is the velocity potential for our original problem. Namely the time derivative $\dot{\phi}^{(k)}$ is the solution $u^{(k)}$ to our original degenerately damped system with initial data $(u_0^{(k)}, 0)$.}

\item Solutions to \eqref{sys:primitive}--\eqref{sys:primitive:ic} decay uniformly to zero at the rate that can be estimated explicitly in terms of the exponent $2m$ of $f$. This is a fundamental example of a dissipative problem with full interior damping. The decay rate can be assessed using, for example, the ODE characterization of \cite{las-tat:93} (see \cite[Coro. 1,  p. 1770]{las-tou:06} for more details; in particular function $h(s)$ there must satisfy $h(s^{2m+2}) \geq s^2$). We get, as $t\to \infty$ that
\[
 E_{\phi^{(k)}}(t) \sim \left(\frac{1}{t}\right)^{1/m}\,.
\]
It follows that the (lower-order) norm $|u^{(k)}(t)|_{_0}^2 =
|\dot{\phi}^{(k)}(t)|_{_0}^2$ decays to zero as $t\to \infty$ with the decay rate uniform with respect to the finite energy  $E_{u^{(k)}}(0)$. By interpolation between $\ltom$ and $H^1(\Om)$ we conclude uniform a decay (but at slower rates, where $1/m$ is modified by the interpolation exponent) for every norm $\|u^{(k)}(t)\|_{W^{s,2}(\Om)}$ with $s\in (0,1)$. 
\begin{remark}
Note that uniform stability of the original system would have required us to include the case $s=1$, which as we are about to show cannot happen.
\end{remark}

\end{enumerate}
\textbf{Summarizing:}  $v^{(k)} = \dot{\phi}^{(k)}$ is precisely the solution $u^{(k)}$ to $\lp f(u^{(k)})\dot{u}^{(k)},  u^{(k)}_0, 0\rp$. Any fractional $W^{s,2}(\Om)$ Sobolev norm of $u^{(k)}$ for $s\in (0,1)$ decays uniformly with respect to $E_{u^{(k)}}(0)$.
In addition, we have the estimate
\begin{equation}\label{L^2-est}
  |u^{(k)}(t)|_{_0}^2 = |\dot{\phi}^{(k)}(t)|_{_0}^2 \leq 2 E_{\phi^{(k)}}(t)
 \leq 2 E_{\phi^{(k)}}(0)  = |\p_x \Phi^{(k)}|_{_0}^2 + |u^{(k)}_0|_{_0}^2 \stackrel{\eqref{elliptic-phi0}}{\leq}  C_2 |u^{(k)}_0|_{_0}^2
\end{equation}
for $C_2$ independent of $u_0^{(k)}$, hence \textbf{independent of $k$} in \eqref{u0-sine}. In particular, the lower order norm $|u^{(k)}(t)|_{_0}$ of the solution to our original problem $\lp f(u)\dot{u}^{(k)}, u^{(k)}_0, u^{(k)}_1 \rp$ is controlled by its initial lower-order norm  $|u^{(k)}_0|_{_0}$ \textbf{independently of $k$}.

Now we are going to use the smallness of the $L^2(\Om)$ norm of $u^{(k)}(t)$ to show that its $\dom(A^{1/2})$ norm cannot decay too fast.

\subsection{Comparison with a conservative problem}
Let $w^{(k)}$ be the unique solution to \emph{linear homogeneous} problem $\lp 0, u^{(k)}_0, 0\rp$. From the energy identity \eqref{energy-identity:0} and by the choice of $u^{(k)}_0$ in \eqref{u0-sine} we get
\[
 E_{w^{(k)}}(t) \equiv 1 = E_{u^{(k)}}(0) \,.
\]
Define 
\[
 z^{(k)} = u^{(k)}-w^{(k)}\,
\]
then $z^{(k)}$ solves   $\lp f(u^{(k)})\dot{u}^{(k)}, 0, 0\rp$. The energy identity \eqref{energy-identity:0} for $z$ gives
\[
  E_{z^{(k)}}(t) = \underbrace{E_{z^{(k)}}(0)}_{=0} - \intt (  f(u^{(k)})\dot{u}^{(k)},\dot{z}^{(k)})_{_0} \quad\leq\quad \intt \|f(u^{(k)})\|_{L^\infty(\Om)} |\dot{u}^{(k)}|_{_0}|\dot{z}^{(k)}|_{_0}
\]

Note that 
\begin{itemize}
\item  $|\dot{u}^{(k)}|_{_0}\cdot |\dot{z}^{(k)}|_{_0} \leq C(\sqrt{E_{u^{(k)}}(0) E_{z^{(k)}}(0)}) \leq 2C$  \quad (since $E_{u^{(k)}}(0) = 1$)  \textbf{independently of $k$} in \eqref{u0-sine}.

\item  Given $f$ as in \eqref{as:f} we have $\|f(u^{(k)})\|_{L^\infty(\Om)}\lesssim  \|u^{(k)}\|_{L^\infty(\Om)}^{2m}$. Via 1-dimensional embeddings  and interpolation for $\eps s <1/2$
\[
 \|u^{(k)}\|_{L^\infty(\Om)}
 \lesssim
  \|u^{(k)}\|_{W^{1-\eps, 2}(\Om)}\lesssim |u^{(k)}|_{_0}^{\eps}\cdot |u^{(k)}|_{_1}^{1-\eps}
\]
so
\[
\begin{split}
 \|u^{(k)}\|_{L^\infty(\Om)}^{2m} \leq & \|u^{(k)}\|_{L^\infty(\Om)}^{2m-1}\cdot |u^{(k)}|_{_0}^{\eps} \cdot |u^{(k)}|_{_1}^{1-\eps} \;\leq \;  C(E_{u^{(k)}}(0),\eps) \, |u^{(k)}|_{_0}^{\eps}
 \end{split}
\]
Because $E_u(0) =1$ independently of $k$ in \eqref{u0-sine}, then by \eqref{L^2-est} we get for any $t\geq 0$
\[
 \|f(u^{(k)}(t))\|_{L^\infty(\Om)} \lesssim \|u^{(k)}(t)\|^{2m}_{L^\infty(\Om)} \leq  C_3 |u^{(k)}_0|_{_0}^\eps.
\]
for $C_3$ \textbf{independent of $k$} in \eqref{u0-sine}.
At this point we finally expand the definition of $u_0$  to get (see \eqref{l2-norm-u0})
\[
\| f(u^{(k)}(t))\|_{L^\infty(\Om)} \leq C_{3,\eps} \frac{1}{k^\eps}\txtforany \eps <1/2 \txtand t\geq 0\,.
\]
\end{itemize}

Plugging these observations into the estimate for $E_{z^{(k)}}(t)$ we obtain: 
\begin{lemma}\label{lem:decay}
 For $u^{(k)}_0(x) \dfn 2 (k\pi)^{-1}\sin(k\pi x)$ let $u^{(k)}$ denote the weak solution on $[0,T]$ to the problem $\lp f(u^{(k)})\dot{u}^{(k)}, u^{(k)}_0, 0\rp$ and $w^{(k)}$ be the solution to linear homogeneous problem $\lp 0, u^{(k)}_0, 0\rp$. Then the difference $z^{(k)}-w^{(k)}$ solves $\lp f(u^{(k)})\dot{u}^{(k)}, 0, 0\rp$ and for any $\eps <1/2$ it satisfies 
\[
 E_{z^{(k)}}(t) \leq  C_\eps T \frac{1}{k^\eps} \txtforall t\in[0,T]\,,
\]
with $C_\eps$ \textbf{independent of} $k$.\qed
\end{lemma}

\subsection{Ruling out uniform stability}
At this point for brevity let us suppress the superscript ``$(k)$" stemming from the choice of the parameter in the definition of initial data. Let $u,w,z$ be as in Lemma \ref{lem:decay}. Suppose for a moment that  $E_z(t)\leq \del$ for some $t$. Then via $ E_w(t) \equiv 1 = E_u(0) $ we have
\begin{equation}\label{E-u-lower-bound}
\begin{split}
E_u(t) = E_{w+z}(t)=& \;\texthalf |w_t(t) + z_t(t) |_{_0}^2  + \texthalf | w_x(t) + z_x(t)|_{_0}^2 \\
=&\;
 E_w(t) + E_z(t) + ( w_t(t), z_t (t))_{_0} + ( w_x(t),  z_x(t) )_{_0}\\ \geq & \; E_w(t) + E_z(t) - 4\sqrt{E_w(t) E_z(t)}
 \geq  1    - 4 \sqrt{\del}\,.
 \end{split}
\end{equation}
To apply this estimate, pick any $\overline{T}>0$ and  find $\del>0$  such that $1-4\sqrt{\del} > 1/2$ (e.g., if $\del < 1/64$). Fix $\eps <1/2$, then there is $k=k(\delta)$ large enough so that  for initial condition $(u_0=2(\pi k)^{-1}\sin(k\pi x), 0)$ yields a solution whose energy satisfies
\[
 E_z(t) \leq C_\eps  \frac{\cl{T}}{k^{\eps}} < \del \txtfor t\in[0,\cl{T}]\,,
\]
for $C_{\eps}$ as in Lemma \ref{lem:decay}. Consequently, by \eqref{E-u-lower-bound},
\[
E_u(t) \geq \frac{1}{2} \txtforall  t\in [0,\cl{T}]\,.
\]
However, the initial condition $(u_0, 0)$  had energy $1$ (again, independently of $k$). Hence the family of initial conditions
\begin{equation}\label{initial-data-frequency}
 \left\{\left( u_0^{(k)}\dfn\frac{2}{k\pi} \sin(k\pi x),\; u_1^{(k)}\dfn 0\right) :  k\in \bbN \right\}
\end{equation}
with the associated solutions $u^{(k)}$, resides in a bounded ball (of radius $\sqrt{2E_{u^{(k)}}(0)}=\sqrt{2}$) in $\scrH^0$, yet the corresponding solutions  do not decay to zero uniformly in the topology of $\scrH^0$.

Thus the associated dynamical system on $\scrH^0$ is not uniformly stable. This argument demonstrates Theorem \ref{thm:unstable} for $\bar{r}=1$ and $r=\half$. The general case follows merely  by attaching a factor of $\sqrt{r}$ to $u^{(k)}_0$ and choosing a potentially  smaller $\delta$ in the last step of the argument.

\subsection{Monotonic strong decay of the energy}

For a weak solution $u$ of $\lp f(u)\dot{u}, u_0, u_1\rp$ with $(u_0,u_1)\in \scrH^0$ the functional
\[
 t\mapsto E_u(t)
\]
is non-increasing. 
We cannot presently claim whether solutions decay to $0$ strongly, however,  it is possible to show that the energy $E_u(t)$ is \emph{strictly}  monotonically decreasing for non-trivial solutions.

Consider a weak solution $u$ on $[0,T]$ and suppose  $E_u(t)\equiv \const$ on $(a,b)\subset [0,T]$, then from the energy identity \eqref{energy-identity:0}  follows that
\[
 \int_a^b (u(t)^{2m}, \dot{u}^2(t))_{_0}  dt=0
\]
Thus $u^m \dot{u} \equiv 0$ a.e. in $(a,b)\times \Om$. Since it is equivalent to $\frac{1}{m+1}\p_t(u^{m+1})=0$, then we conclude that $u^{m+1}(t_1)=u^{m+1}(t_2)$ in $\ltom$ for a.e. $t_1,t_2$:
\[
  u^{m+1}  = \const \txtin \ltom \quad \text{ a.e. } \quad t\in (a,b).
\]
Moreover, since $u\in C_T \V \into C ([0,T]\times \cl{\Om})$, then we have   $u(t, x) = \pm u_0(x)$ for   every $(t,x)\in (a,b)\times \Om$. By the continuity of the solution this  is only possible if $u \equiv u_0$ for $(t,x)\in (a,b)\times \Om$. Then $\dot{u}\equiv 0$ and we arrive at an equilibrium solution which has to be trivial. This observation completes the proof of Theorem \ref{thm:unstable}.

\section{Numerical results}\label{sec:numerics}

The strategy for the proof of instability was largely prompted by numerical observations described below.

\subsection{Outline of the numerical approach}
The  numerical implementation presented here treats the case $m=1$ of \eqref{as:f}:
\[
  \ddot{u} - u_{xx}  +  u^2  \dot{u} = 0
\]
with
\[
 u(t,0) \equiv u(t,1) \equiv 0
\]
and given initial data 
\[
 u(0,x) = u_0(x) \txtand  \dot{u}(0,x) = u_1(x)\,.
\]
Solution was discretized in space via a Ritz-Galerkin method. The dynamic problem could be analyzed explicitly using a discretization in time  and a Runge-Kutta scheme, though, rigorous justification of convergence becomes more delicate. Another approach is to approximate the successive approximations   of Theorem \ref{thm:local} which, when exact, are guaranteed to converge, at least over small time intervals. The iterates correspond to \emph{linear} inhomogeneous PDE problems that are resolved using a hybrid scheme:
\begin{enumerate}[(i)]
	\item for relatively short times find solutions using an approximation of semi-discrete Ritz-Galerkin method by discretizing time-integrals in the variation of parameter formula. If the error in numerical integration is small, then this approach enjoys an explicit convergence estimate (for smooth solutions and over finite time intervals) essentially proportional to the space discretization parameter $h$.

	\item for larger times, collect the last $k$-points of the semi-discrete approximation and resolve the rest of the iteration using a multi-step method (Adams-Bashforth).
\end{enumerate} 

 Thus, we begin  with some initial guess $(u_{[0]}, \dot{u}_{[0]})$ and proceed to solve inhomogeneous linear problems
\begin{equation}\label{k-th-problem}
\ddot{u}_{[k]} - \p_{xx} u_{[k]}  =  F_{[k-1]}
\end{equation}
on the space $\scrH^0=\V \times \ltom$ with the  forcing term from the preceding iteration
\[
 F_{[k-1]} = - u_{[k-1]}^2 \dot{u}_{[k-1]}\,.
\]
The constants in the estimates \eqref{Lambda-contr} and \eqref{T-invariance} could potentially be determined explicitly in this case  (by following the proof with  specific $\al$ and $m$ in the definition \eqref{as:f} of $f$)  which would yield an explicit bound on the Lipschitz constant ``$\gam$" of the contraction mapping in terms of $T$. In turn, given this constant $\gam>0$  if
\[
  \|u_{[k]}-u_{[k-1]}\|_{C_T \scrH^0} < \eps
\]
then the absolute error between $k$-th iteration and the true solution is no more than $\eps \gam/(1-\gam)$.

If we use a Ritz-Galerkin scheme with element size $h$ to find an approximate solution $u_{[k,h]}$ to linear inhomogeneous problem \eqref{k-th-problem}, then for $h$ small, e.g. see \cite[Thm. 13.1, p. 202]{b:lar-tho}, and for simplicity taking the initial conditions to be the more accurate Ritz projections of the initial data, we get
\[
 \|u_{[k]} - u_{[k,h]} \|_{C_T\scrH^0} \leq     C h  \intT |\ddot{u}_{[k]}(s)|_2 ds\,.
\]
This estimate of course requires sufficiently regular solutions.
As Theorem \ref{thm:well-posed} and Example \ref{ex:square} show, in order to have the $L^2 W^{2,2}(\Om)$  regularity on $\ddot{u}_{[k]}$  it suffices to have initial data in the space
\[
 (u(0), \dot{u}(0)) \in \scrH^2 = \dom(A^{2})\times \dom(A^{3/2})\,.
\]
For example, the demonstrated numerical results below use displacement and velocity proportional to the eigenfunctions of $A$, which are smooth.

\subsection{Specifics of the implementation}\label{particulars}

Consider an equipartitioned mesh of subintervals of length $h$ and the standard  piecewise linear nodal basis $\{ \phi^h_i\}$, with $i=1,\ldots N\dfn (h^{-1}-1)$. Let  $\bbA_h$ denote the restriction of the linear evolution generator $\bbA$ to the subspace $\scrH_h^0$ of $\scrH$  spanned by $\{ \psi^{(h)}_{ij} = (\phi_i^{(h)}, \phi_j^{(h)})\}$. By $\cR^{1}_h$ denote the elliptic Ritz projection on the subspace of $\V$ and  let $\cR^{0}_h$ stand for the corresponding $L^2(\Om)$ projection.

Given a candidate approximation 
\[
y_{[k-1,h]} = \vecn{ u_{[k-1,h]}\\ \dot{u}_{[k-1,h]}}
\]
we compute the forcing 
\[
f_{[k-1,h]}(t,x) = u_{[k-1,h]}^2(t,x)\dot{u}_{[k-1,h]}(t,x)
\]
The coefficient vector $\bff_{[k,h]}$ of the projection of $\cR_h^0 f_{[k,h]}$ is obtained in terms of the coefficients
$\ds C_{pqrs} = \intOm \phi^h_p(x) \phi^h_q(x) \phi^h_r(x) \phi^h_s(x)dx $  (which  for this choice of basis functions $\phi^h_i$ form a very sparse tensor with only 3 distinct values).
The initial guess used to calculate $f_{[0,h]}$ is the constant solution :
\[
 f_{[0,h]}(t,x) =  (\cR_h^{1} u_0(x))^2(\cR_h^{0} u_1(x)) \txtforall t\geq 0.
\]
Let $\bfy_0$ denote the projection of the initial data  $(\cR^1_h u_0, \cR_h^0 u_1)$. We obtain a semi-discrete approximation of the original system \system{} for unknown coefficient vector $\bfy_{[k,h]}$:
\[
\bfy_{[k,h]}'  = -\bbA_h \bfy_{[k,h]} +
\bff_{[k-1,h]},\qquad \bfy(0) = \bfy_0\,.
\]
For relatively short times we can invoke
\[
\bfy_{k,h}(t) =  e^{-t \bbA_h} \bfy_0 + I_t \txtwith I_t \dfn \int_0^t e^{-\bbA_h (t-s)}\bff_{[k-1,h]}(s)ds
\]

which is in turn discretized over time scale $\bfT=(t_1,t_2,\ldots, t_N)$  with $t_{i+1}-t_i = \del$. According to
\[
I_t  =
 e^{-\bbA_h(t-\bar{t})} I_{\bar{t}} + 
 \int_{\bar{t}}^{t} e^{-\bbA_h (t-s)} \bff_{[k-1,h]}(s)ds\,.
\]
at each step only the integral over $[\bar{t},t]$ needs to be computed. For this purpose only several values of the matrix exponentials $e^{-\bbA_h(t-s)}$ are needed in order to apply the Newton-Cotes rule on sub-interval $[\bar{t}, t]$, specifically:
\[
e^{j\cdot \delta \bbA_h} \txtfor j=1,2,\ldots,m-1
\]
where $m$ is number of points used for Newton-Cotes formula (e.g., Boole's  or Simpson's $3/8$th).  These $m-1$ matrices need to be computed just once and only depend on the time-step, but not the total number of these steps.

In turn, the vectors  $e^{-t_k \bbA_h} \bfy_0$ have to be found for each $t_k$. But since $\bfy_0$ is fixed, these  can be more accurately determined using scaling and truncated Taylor series approximation \cite{alm-hih:11:SIAMJSC}.

For simulations over larger time-scales we can use the last few values:
\[
 (t_{N-p}, \bfy_{k,h}(t_{N-p})), \cdots,  (t_N, \bfy_{k,h}(t_N))
\]
to initialize  a linear $p$-step method, e.g., $5$-step 
Adams-Bashforth to efficiently obtain solution on the interval
 $[t_N, t_{\text{final}}]$.

\subsection{Pointwise Runge-Kutta solutions for particular initial data}\label{sec:pointwise-rk}

As before, let $E_k$ be the eigenfunction $\sqrt{2} \sin(\pi k x)$ for $A$ with eigenvalue $\lam_k = \pi^2 k^2$. Then for initial data
\begin{equation}\label{ic-eigenfunctions}
 u_0(x) = c_0 E_k(x) \txtand u_1 = c_1 E_k(x)
\end{equation}
for constants $c_0$, $c_1$, the solution of \system{} can be reduced to a dissipative ODE using the ansatz
\begin{equation}\label{ansatz}
 u(t,x) = \phi(t) E_k(x)\,.
\end{equation}
Plugging it into \eqref{wave}  equation yields
\[
 \phi'' E_k + \lam_k \phi E_k +  E_k^2 \phi^2 \phi' = 0
\]
This identity would be implied if for each $x\in \Om$ function  $\phi$ solves the 2nd-order nonlinear ODE
\[
\phi''  + \lam_k \phi + E_k(x) \phi^2 \phi' = 0\,.
\]
\[
 \phi(0)=c_1\txtand \phi'(0)=c_2.
\]
It corresponds to a first-order nonlinear system:
\begin{equation}\label{ode-ansatz}
\Psi' = \bbF_k(\Psi, x), \qquad \bbF_k\left(\vecn{\al\\\bet},x\right) \dfn \vecn{\bet\\ -\lam_k \al - E_k(x) \al^2 \bet},\quad \Psi(0) =\vecn{c_0\\c_1}\,.
\end{equation}
Function $\bbF_k$ is smooth with respect to the components of $\Psi$ and to variable  $x$, which now acts as a parameter. This ODE system has global differentiable solutions, moreover since $E_k(x)$ is smooth, in fact, analytic in $x$ then local solutions are differentiable with respect to $x$ \cite[Thm. 3.1, p. 95]{b:hartman:02}.

Because the initial data \eqref{ic-eigenfunctions} is smooth then by Theorem \ref{thm:well-posed} the unique solution $u$ is, among other things, in $C^1_T C^1(\cl{\Om})$.  Consequently $u$ must coincide with the solution to the ansatz  
\eqref{ansatz}.

In turn, \eqref{ode-ansatz} is a dissipative $2\times 2$ system of ODEs and can be approximated by a  Runge-Kutta scheme. To get some quantitative estimate on the absolute error of solutions found in Section \ref{particulars}, at least for initial data of the form \eqref{ic-eigenfunctions}, one can consider a piecewise linear interpolation of \eqref{ode-ansatz} and then calculate the energy-norm difference from the finite-element solution.

\subsection{Energy plots}

The accompanying figures and data demonstrate some of the numerical results. The initial data is considered of the form
\[
 u_0(x) = \frac{2}{\pi k}  \sin (k\pi x),\quad u_1\equiv 0
\]
which permits to compare the finite element solutions  to the pointwise Runge-Kutta solutions described in Section \ref{sec:pointwise-rk}.

Figure \ref{fig:1} shows the  point-value of displacement $u(t,0.5)$ next to the displacement value at the same $x=0.5$ for the  corresponding initial boundary value problem \emph{with linear damping}.

\def\captionText{Time interval:  $[0,T=10]$. Obtained by approximating the Ritz-Galerkin semi-discrete solution using numerical integration in the variation of parameter formula. Element size: $h=10^{-2}$; time-step: $\delta=2\cdot 10^{-3}$.}

\def\captionTextExt{Time interval:  $[0,T=50]$. On the interval $[0,10]$ obtained by approximating the Ritz-Galerkin semi-discrete solution using numerical integration in the variation of parameter formula, with element size $h=10^{-2}$ and time-step $\delta=2\cdot 10^{-3}$. Extended to  $t\in [10,50]$ using $5$-step Adams-Bashforth method.
}

Figure \ref{fig:2} presents  numerical estimates of the energy for solutions obtained by Ritz-Galerkin finite element scheme and successive approximations. The graphs indicate that  the energy decay deteriorates as the frequency of the initial data goes up while the initial finite-energy remains fixed ($E_{u^{(k)}}(0)=1$ independently of $k$), thus illustrating the lack of uniform which was  rigorously confirmed by Theorem \ref{thm:unstable}. 
The initial data are of the form \eqref{ic-eigenfunctions} with zero initial velocity. The indicated errors are obtained by comparing  each finite-element solution to a piecewise-linear interpolant of the corresponding piecewise RK solution \eqref{ansatz}.

Figure \ref{fig:3} uses multi-step extensions of the same solutions shown in Figure \ref{fig:2} to a larger time-scale using (5-step) Adams-Bashforth method. It also includes the decay of the $\ltom$-norm $|u^{(k)}(t)|_{_0}$ for these solutions.

\begin{figure}
\centering
\includegraphics[width=.8\textwidth, keepaspectratio=true, trim=1.5in 0in 1.5in 0in, clip=true]{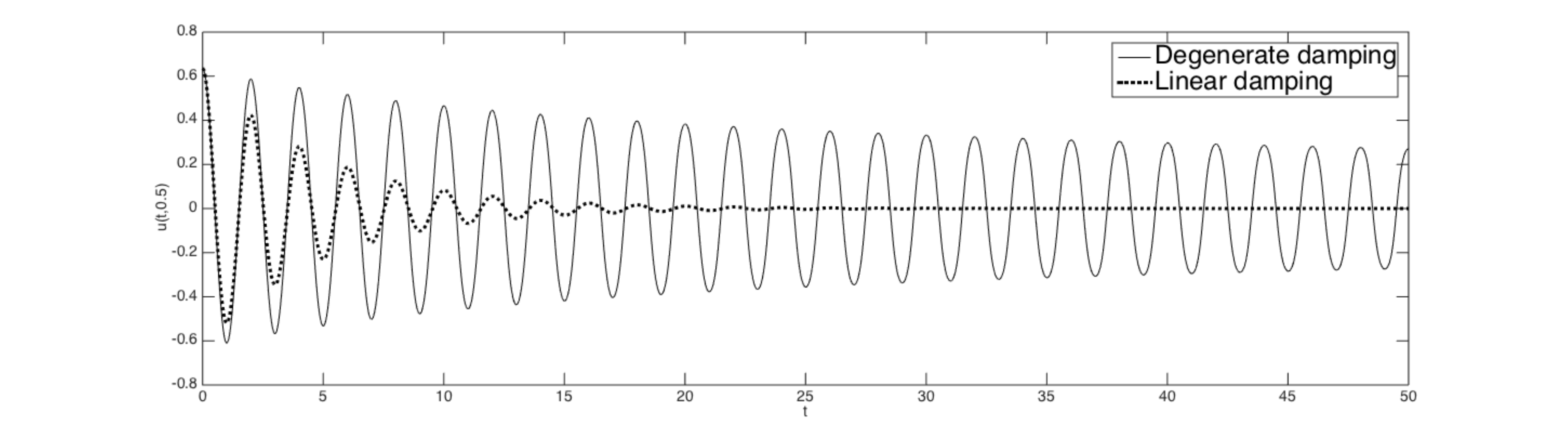}
\caption{The displacement value $u(t,x=0.5)$ of a numerical solution to problem \system{}  with $f(s)=s^2$. Initial displacement $u_0(x) = 2\pi^{-1}\sin(\pi x)$, initial velocity $u_1(x)=0$. 
\captionText{}
 Displayed next to the exact analytic solution of the corresponding initial boundary value problem with linear damping: $\ddot{u}  - u_{xx}+ (2/\pi)^2 \dot{u} = 0$.
}\label{fig:1}
\end{figure}

\begin{figure}
\centering
\includegraphics[width=.8\textwidth, keepaspectratio=true, trim=1.5in 0in 1.5in 0in, clip=true]{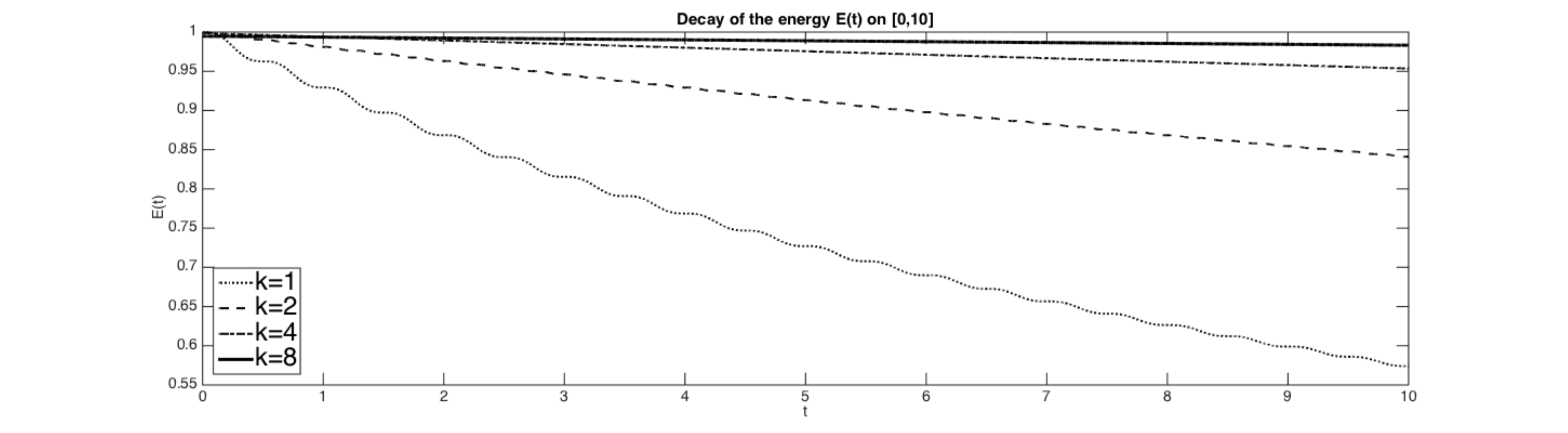}
\caption{Plots of energy $E_{u^{(k)}}(t)$ versus time $t$ for system \system{} with $f(s)=s^2$. \captionText{}
Initial displacements:   $u^{(k)}_0 = 2 (k\pi)^{-1} \sin(k \pi x)$ for frequencies $k=1,2,4,8$. Initial velocities are zero. 
Approximate error was obtained by computing maximal, over $[0,T]$, energy ($\scrH^0$-norm) difference from a piecewise-linear interpolation of the  pointwise RK solution given by ansatz \eqref{ansatz}. The initial energy of each solution is $1$. Corresponding errors $e_k$ are:
$e_1=3.84$E-$02$,
$e_2=5.20$E-$03$,
$e_4=6.97$E-$03$,
$e_8=4.15$E-$01$. 
}
\label{fig:2}
\end{figure}

\begin{figure}
\centering
\includegraphics[width=.8\textwidth, keepaspectratio=true, trim=1.5in 0in 1.5in 0in, clip=true]{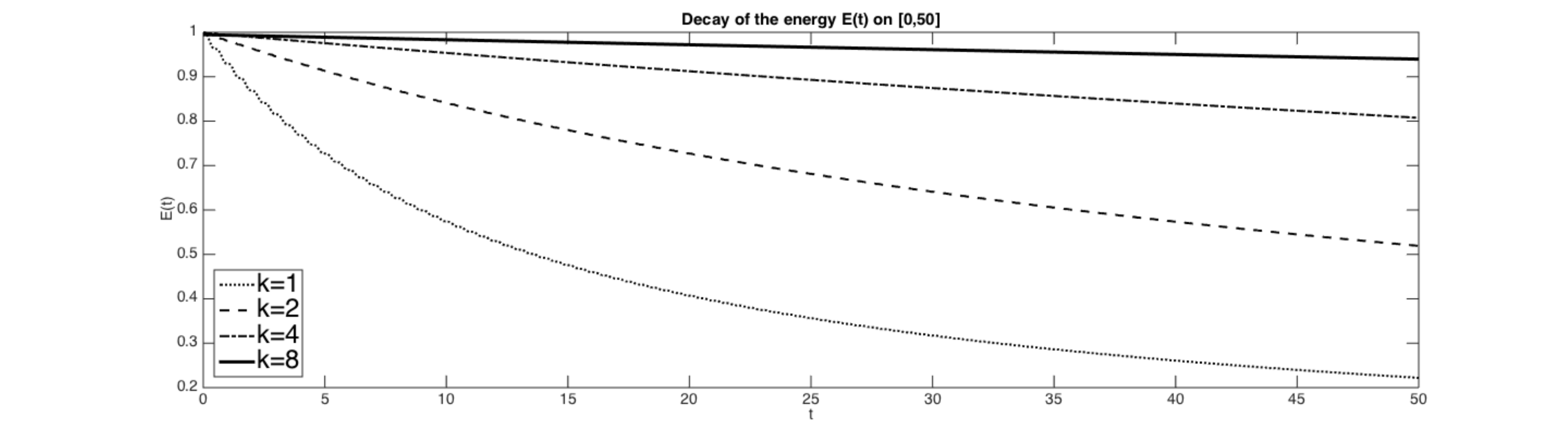}\\
\includegraphics[width=.8\textwidth, keepaspectratio=true, trim=1.5in 0in 1.5in 0in, clip=true]{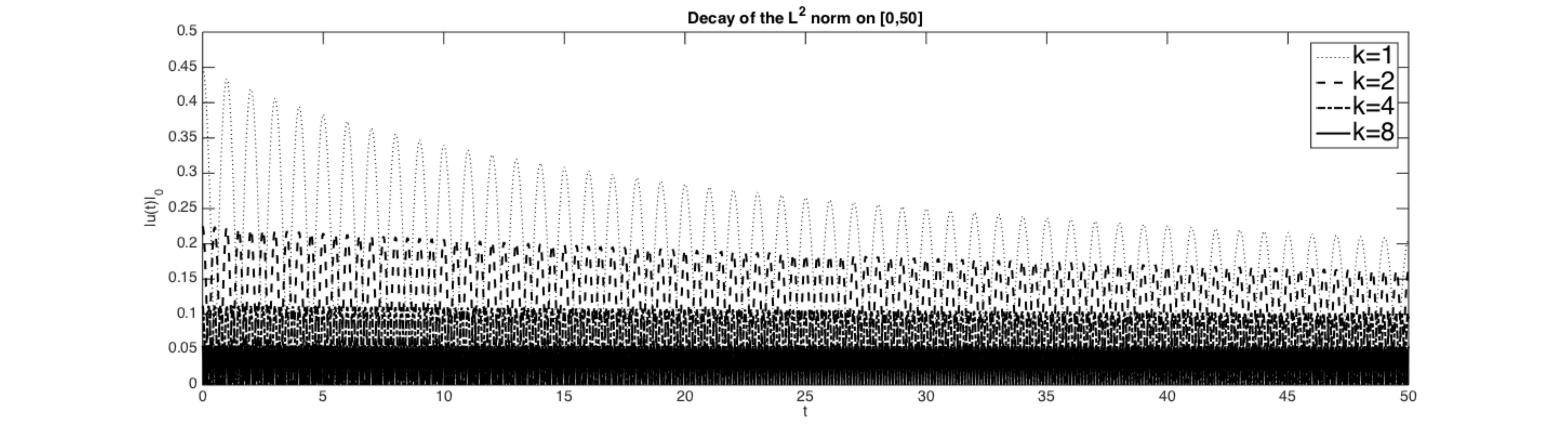}
\caption{Plots of energy $E_{u^{(k)}}(t)$ and $L^2(\Om)$ norm $|u^{(k)}(t)|_{_0}$   vs time $t$  for a numerical solution of  problem \system{}  with $f(s)=s^2$. This is an extension of the solutions from Figure \ref{fig:2} (originally defined for $t\in[0,10]$) to the time-interval $[0,50]$ by $5$-step Adams-Bashforth method. Due to rapid oscillations, the $\ltom$ norm values for the solution corresponding to $k=8$ appear to fill out a solid region.
}\label{fig:3}
\end{figure}

\section*{Acknowledgments}
The research of George Avalos and Daniel Toundykov was partially supported by the National Science Foundation grant DMS-1211232.
The numerical analysis presented in this work and partial theoretical results were obtained during a  Research Experience for Undergraduates (REU) program on applied partial differential equations in summer 2013 at the University of Nebraska-Lincoln. This REU site was funded by the National Science Foundation under grant NSF DMS-1263132. The REU project would not have been possible without assistance of Tom Clark.

\appendix 
\section{Finite-dimensional counterpart}
To complement the analysis of the infinite-dimensional model \system{} it is also interesting to examine the related finite-dimensional version of a degenerately damped harmonic oscillator:
\begin{equation}\label{ode}
  \ddot{x} + k x  + f(x) \dot{x} = 0,\quad (k>0)
\end{equation}
\begin{equation}\label{ode:ic}
  x(0)  = x_0,\quad \dot{x}(0)=x_1\,
\end{equation}
for $f = \al s^{2m}$, $m\in \bbN$. Rewrite it  as a first order evolution problem
\begin{equation}\label{ode-y}
 \dot{y} = G(y) \txtfor  G(y)\dfn \vecn{ y_2  \\  -f(y_1) y_2 - k y_1},\quad y(0) = \vecn{x_0\\x_1}\,.
\end{equation}
Henceforth, let $|\cdot|$ denote an equivalent norm on $\bbR^2$:
 \begin{equation}\label{equiv-norm}
  |v|^2 \dfn \texthalf k |v_1|^2 + \texthalf |v_2|^2\,.
 \end{equation}
Because $f$ is smooth and  non-negative, then  classical ODE results guarantee that   solutions are unique, exist globally and satisfy
\[
| y(t)|\leq |y(0)| \txtforall t\geq 0\,.
\]
Below we present stability results which contrast their  infinite-dimensional analogues discussed earlier. In particular, the finite-dimensional system is uniformly stable, while for distributed-parameter version the strong stability is open while uniform stability has been proven false.

\begin{lemma}\label{lem:lasalle}
The dynamical system corresponding to the ODE \odesystem{} is asymptotically (``strongly") stable.
\end{lemma}

\begin{proof}
This is a direct consequence of LaSalle's invariance principle with the Lyapunov function given by the equivalent norm \eqref{equiv-norm}: $V(y)= |y|^2$. We only need to check what kinds of trajectories reside in the invariant set of the system:
\[
E =\{ y\in \bbR^2  :  \nabla V(y) \cdot G(y) = 0 \}\,.
\]
So for $y\in E$
\[
k y_1 y_2  - f(y_1) y_2^2 - ky_1 y_2 = 0\,.
\]
\[
f(y_1)y_2^2 = 0
\]
In particular, either $y_1=0$ or $y_2 =0$.

Now suppose that a solution trajectory $\{(t, y=(x(t),\dot{x}(t))) : t\geq 0\}$ resides in $E$. If at some $t$ we have $x(t)\neq 0$, then by the continuity in time, it is nonzero on some interval $I$. On that interval we must have $\dot{x}\equiv 0$ by the property of $E$. But then on that interval, from equation \eqref{ode} we get a contradiction since the solution has to be constant and therefore zero.

Thus the only trajectory in $E$ is the trivial one.
By LaSalle's invariance principle every solution  is asymptotically  stable.
\end{proof}

\begin{theorem}
The dynamical system corresponding to the ODE \odesystem{} is uniformly stable.
\end{theorem}

\begin{proof} 
Proceed by contradiction. Assume for a bounded set $B\subset \bbR^2$ there exists some $\eps>0$,  a bounded sequence of initial data $(y_{0n})_n\subset B$, and a sequence of corresponding times $(T_n)_n$ with $T_n\nearrow \infty$ such that
\[
|y_n(T_n)|>\eps\,.
\]
Extract a convergent subsequence of initial data, reindexed again by $n$. Let $z_0 \in \cl{B}$ denote this limit point and $t\mapsto z(t)$ be the corresponding solution. By Lemma \ref{lem:lasalle}  
\[
\lim_{t\to\infty}z(t)=\bfzer\,.
\]
In particular, there exists $T$  such that  for  $t\geq T$ we have  $|z(t)|< \eps/2$.  Because the non-linearity $(x_1,x_2)\mapsto f(x_1)x_2$ is locally Lipschitz on $\bbR^2$, and the system \eqref{ode-y} is non-accretive, then there exists $\del=\del(T, \cl{B})>0$  such that for any $\eta_0 \in B$ if $|\eta-z_0|<\del$, the corresponding solutions satisfy
\[
|\eta(t)-z(t)|<\frac{\eps}{2} \txtfor t\in[0,T]\,.
\]
Since $T_n \nearrow \infty$, we can find $T_N > T$. Next, because $y_{0n}$ converge to $z_0$, then  for $n> N$ we can find find $y_{0n}$ so that   $|y_{0n}-z_0|<\del$  and, consequently,
\[
 |y_n(t)-z(t)| < \frac{\eps}{2}\txtfor t\in [0,T] \,.
\]
Then $|y_n(T)| < \eps/2  + |z(T)| < \eps$. Because the system is non-accretive, then  $|y_n(t)| < \eps$ for all $t\geq T$, and in particular it holds for $T_n \geq T_N > T$. So $|y_n(T_n)| <\eps$ which contradicts the choice of $y_{0n}$.
\end{proof}


\begin{thebibliography}{10}

\bibitem{van-mar:00}
Judith Vancostenoble and Patrick Martinez.
\newblock Optimality of energy estimates for the wave equation with nonlinear
  boundary velocity feedbacks.
\newblock {\em SIAM J. Control Optim.}, 39(3):776--797 electronic, 2000.

\bibitem{nicaise:03}
S.~Nicaise.
\newblock Stability and controllability of an abstract evolution equation of
  hyperbolic type and concrete applications.
\newblock {\em Rend. Mat. Appl. (7)}, 23(1):83--116, 2003.

\bibitem{las-tou:06}
Irena Lasiecka and Daniel Toundykov.
\newblock Energy decay rates for the semilinear wave equation with nonlinear
  localized damping and source terms.
\newblock {\em Nonlinear Anal.}, 64(8):1757--1797, 2006.

\bibitem{nic-pig:05:AAA}
Serge Nicaise and Cristina Pignotti.
\newblock Internal stabilization of {M}axwell's equations in heterogeneous
  media.
\newblock {\em Abstr. Appl. Anal.}, (7):791--811, 2005.

\bibitem{las-tri:97}
Irena Lasiecka and Roberto Triggiani.
\newblock {\em Carleman estimates and exact boundary controllability for a
  system of coupled, nonconservative second-order hyperbolic equations}, volume
  188 of {\em Partial differential equation methods in control and shape
  analysis (Pisa); Lecture Notes in Pure and Appl. Math.}, pages 215--243.
\newblock Dekker, 1997.

\bibitem{tri-yao:02:AMO}
Roberto Triggiani and Peng-Fei Yao.
\newblock Carleman estimates with no lower-order terms for general {R}iemann
  wave equations. {G}lobal uniqueness and observability in one shot.
\newblock {\em Appl. Math. Optim.}, 46(2-3):331--375, 2002.
\newblock Special issue dedicated to the memory of Jacques-Louis Lions.

\bibitem{las-tri-zha:04:JIIP-I}
Irena Lasiecka, Roberto Triggiani, and X.~Zhang.
\newblock Global uniqueness, observability and stabilization of nonconservative
  {S}chr{\"o}dinger equations via pointwise {C}arleman estimates. {I}.
  {$H^1(\Omega)$}-estimates.
\newblock {\em J. Inverse Ill-Posed Probl.}, 12(1):43--123, 2004.

\bibitem{las-tri-zha:04:JIIP-II}
Irena Lasiecka, Roberto Triggiani, and X.~Zhang.
\newblock Global uniqueness, observability and stabilization of nonconservative
  {S}chr{\"o}dinger equations via pointwise {C}arleman estimates. {II}.
  {$L_2(\Omega)$}-estimates.
\newblock {\em J. Inverse Ill-Posed Probl.}, 12(2):183--231, 2004.

\bibitem{tri-xu:07:AMS}
Roberto Triggiani and Xiangjin Xu.
\newblock Pointwise {C}arleman estimates, global uniqueness, observability, and
  stabilization for {S}chr\"odinger equations on {R}iemannian manifolds at the
  {$H^1(\Omega)$}-level.
\newblock In {\em Control methods in {PDE}-dynamical systems}, volume 426 of
  {\em Contemp. Math.}, pages 339--404. Amer. Math. Soc., Providence, RI, 2007.

\bibitem{chu-las-tou:08:DCDS}
Igor Chueshov, Irena Lasiecka, and Daniel Toundykov.
\newblock Long-term dynamics of semilinear wave equation with nonlinear
  localized interior damping and a source term of critical exponent.
\newblock {\em Discrete Contin. Dyn. Syst.}, 20(3):459--509, 2008.

\bibitem{buc-tou:10}
Francesca Bucci and Daniel Toundykov.
\newblock Finite-dimensional attractors for systems of composite wave/plate
  equations with localized damping.
\newblock {\em Nonlinearity}, 23(9):2271--2306, 2010.

\bibitem{bar-leb-rau:92}
Claude Bardos, Gilles Lebeau, and Jeffrey Rauch.
\newblock Sharp sufficient conditions for the observation, control, and
  stabilization of waves from the boundary.
\newblock {\em SIAM J. Control Optim.}, 30(5):1024--1065, 1992.

\bibitem{leb-rob:97:DMJ}
Gilles Lebeau and Luc Robbiano.
\newblock Stabilisation de l'\'equation des ondes par le bord.
\newblock {\em Duke Math. J.}, 86(3):465--491, 1997.

\bibitem{leb-rob:95:CPDE}
G.~Lebeau and L.~Robbiano.
\newblock Contr\^ole exact de l'\'equation de la chaleur.
\newblock {\em Comm. Partial Differential Equations}, 20(1-2):335--356, 1995.

\bibitem{bellassoued:05}
Mourad Bellassoued.
\newblock Decay of solutions of the wave equation with arbitrary localized
  nonlinear damping.
\newblock {\em J. Differential Equations}, 211(2):303--332, 2005.

\bibitem{bor-tom:10:MA}
Alexander Borichev and Yuri Tomilov.
\newblock Optimal polynomial decay of functions and operator semigroups.
\newblock {\em Math. Ann.}, 347(2):455--478, 2010.

\bibitem{ell-tou:12:EECT}
Matthias Eller and Daniel Toundykov.
\newblock Carleman estimates for elliptic boundary value problems with
  applications to the stabilization of hyperbolic systems.
\newblock {\em Evol. Equ. Control Theory}, 1(2):271--296, 2012.
\newblock Evolution Equations and Control Theory.

\bibitem{ram-stre:02:TAMS}
Mohammad~A. Rammaha and Theresa~A. Strei.
\newblock Global existence and nonexistence for nonlinear wave equations with
  damping and source terms.
\newblock {\em Trans. Amer. Math. Soc.}, 354(9):3621--3637 (electronic), 2002.

\bibitem{pit-ram:02:IUMJ}
David~R. Pitts and Mohammad~A. Rammaha.
\newblock Global existence and non-existence theorems for nonlinear wave
  equations.
\newblock {\em Indiana Univ. Math. J.}, 51(6):1479--1509, 2002.

\bibitem{bar-las-ram:05}
Viorel Barbu, Irena Lasiecka, and Mohammad Rammaha.
\newblock On nonlinear wave equations with degenerate damping and source terms.
\newblock {\em Trans. Amer. Math. Soc.}, 357(7):2571--2611 (electronic), 2005.

\bibitem{bar-las-ram:07:IUMJ}
Viorel Barbu, Irena Lasiecka, and Mohammad~A. Rammaha.
\newblock Blow-up of generalized solutions to wave equations with nonlinear
  degenerate damping and source terms.
\newblock {\em Indiana Univ. Math. J.}, 56(3):995--1021, 2007.

\bibitem{las-tat:93}
Irena Lasiecka and Daniel Tataru.
\newblock Uniform boundary stabilization of semilinear wave equations with
  nonlinear boundary damping.
\newblock {\em Differential Integral Equations}, 6(3):507--533, 1993.

\bibitem{ram-tou-wil:12:DCDS}
Mohammad~A. Rammaha, Daniel Toundykov, and Wilstein Zahava.
\newblock Global existence and decay of energy for a nonlinear wave equation
  with $p$-laplacian damping.
\newblock {\em Discrete Contin. Dyn. Syst.}, 32(12):4361--4390, 2012.

\bibitem{b:barbu93:rom}
Viorel Barbu.
\newblock {\em Probleme la limit\u a pentru ecua\c tii cu derivate par\c
  tiale}.
\newblock Analiz\u a Modern\u a \c si Aplica\c tii. [Modern Analysis and
  Applications]. Editura Academiei Rom\^ane, Bucharest, 1993.

\bibitem{b:barbu:93}
Viorel Barbu.
\newblock {\em Analysis and control of nonlinear infinite-dimensional systems},
  volume 190.
\newblock Academic Press Inc, Boston, MA, 1993.

\bibitem{b:lar-tho}
Stig Larsson and Vidar Thom\'{e}e.
\newblock {\em Partial differential equations with numerical methods},
  volume~45 of {\em Texts in Applied Mathematics}.
\newblock Springer-Verlag, Berlin, 2008.

\bibitem{alm-hih:11:SIAMJSC}
A.~Al-Mohy and N.~Higham.
\newblock Computing the action of the matrix exponential, with an application
  to exponential integrators.
\newblock {\em SIAM Journal on Scientific Computing}, 33(2):488--511, 2011.

\bibitem{b:hartman:02}
Philip Hartman.
\newblock {\em Ordinary differential equations}, volume~38 of {\em Classics in
  Applied Mathematics}.
\newblock Society for Industrial and Applied Mathematics (SIAM), Philadelphia,
  PA, 2002.
\newblock Corrected reprint of the second (1982) edition [Birkh{\"a}user,
  Boston, MA; MR0658490 (83e:34002)], With a foreword by Peter Bates.

\end{thebibliography}

\end{document}